\documentclass[lineno]{biometrika}
\usepackage{amsmath}
\usepackage{color}

\usepackage{newtxtext}
\usepackage[subscriptcorrection]{newtxmath}

\graphicspath{{./art/}}

\usepackage[plain,noend]{algorithm2e}

\makeatletter
\renewcommand{\algocf@captiontext}[2]{#1\algocf@typo. \AlCapFnt{}#2} 
\def\@algocf@capt@plain{top}
\renewcommand{\algocf@makecaption}[2]{%
  \addtolength{\hsize}{\algomargin}%
  \sbox\@tempboxa{\algocf@captiontext{#1}{#2}}%
  \ifdim\wd\@tempboxa >\hsize
    \hskip .5\algomargin%
    \parbox[t]{\hsize}{\algocf@captiontext{#1}{#2}}
  \else%
    \global\@minipagefalse%
    \hbox to\hsize{\box\@tempboxa}
  \fi%
  \addtolength{\hsize}{-\algomargin}%
}
\makeatother

\newcommand \ag{\alpha}

\newcommand \ga{\gamma}
\newcommand \Gg{\Gamma}

\newcommand \eg{\varepsilon}
\newcommand \ep{\epsilon}

\newcommand \la{\lambda}

\newcommand \sg{\sigma}

\newcommand{\thg}{\theta}

\newcommand\bzero{\mbox{\boldmath${0}$}}

\newcommand\bx{{\bf x}}

\newcommand\bz{{\bf z}}
\newcommand\bZ{{\bf Z}}

\newcommand\mbB{{\mathbb B}}

\newcommand\mbE{{\mathbb E}}

\newcommand\mbR{{\mathbb R}}
\newcommand\mbS{{\mathbb S}}

\newcommand\cA{{\mathcal A}}

\newcommand\cC{{\mathcal C}}

\newcommand\cT{{\mathcal T}}

\newcommand\cZ{{\mathcal Z}}

\newcommand\lp{\left (}
\newcommand\rp{\right )}
\newcommand\lb{\left [}
\newcommand\rb{\right ]}
\newcommand\lmo{\left |}
\newcommand\rmo{\right |}
\newcommand\lbr{\left \{}
\newcommand\rbr{\right \}}
\newcommand\lip{\left \langle}
\newcommand\rip{\right \rangle}

\def\convP{\stackrel{\mbox{$\scriptstyle P$}}{\rightarrow}}
\def\convv{\stackrel{\mbox{$\scriptstyle v$}}{\rightarrow}}
\def\convd{\stackrel{\mbox{$\scriptstyle d$}}{\rightarrow}}
\newcommand\var{\hbox{\rm var}}

\begin{document}



\markboth{M. Kim and P. Kokoszka}{Extremal correlation coefficient for functional data}

\title{Extremal correlation coefficient for functional data}

\author{M. Kim}
\affil{Department of Statistics, West Virginia University,\\ Morgantown, 26505, WV, USA.
\email{mihyun.kim@mail.wvu.edu}}

\author{P. Kokoszka}
\affil{Department of Statistics, Colorado State University,\\ Fort Collins, 80523, CO, USA \email{piotr.kokoszka@colostate.edu}}

\maketitle

\begin{abstract}
We propose a coefficient that measures extremal dependence in
paired samples of functions. It has properties similar to the
Pearson correlation,  but  differs in significant ways:
(i) it is designed to measure dependence between curves, (ii)
it focuses only on extreme curves. The new coefficient is derived
within the framework of regular variation in Banach
spaces. A consistent estimator is proposed and justified
by an asymptotic analysis and a simulation study.
The usefulness of the new coefficient is illustrated using
financial and climate functional data.
\end{abstract}


\section{Introduction}
With the growing impact of extreme events such as financial downturns or unusual weather, there has been increasing interest in developing statistical tools to study patterns of extreme curves. This is to a large extent due to the increasing availability of high resolution data; asset price curves can be constructed at any temporal resolution, and modern weather databases and computer models contain measurements at hourly or even higher frequencies. Such data can be interpreted as curves, e.g., one curve per day, providing a more comprehensive view of daily patterns compared to a single summary number like the closing price or maximum temperature. Analyzing extreme curves in the framework of functional data analysis thus leads to a more precise understanding of the impacts associated with extreme events.

This paper makes a methodological and theoretical contribution at the nexus of extreme value theory and functional data analysis. We propose a coefficient that quantifies the tendency of paired extreme curves to exhibit similar patterns simultaneously. Two examples of the type of questions that the tool deals with are the following: (i) during a stock market crisis, such as the market decline  due to the covid-19 pandemic, do
returns  of  different sectors of the economy exhibit similar extreme daily trajectories? (ii) how likely is location A to experience  a similar daily pattern of temperature  as location B (on the same day)  during a heat wave? Our proposed coefficient offers a more precise quantification of extreme risk by focusing on (i) the shape of curves and (ii) the extreme parts of paired samples. This point is further illustrated with a data example in Section~\ref{ss:etf}.

There has been some research focusing on probabilistic and
statistical methods for extreme curves.
Extreme value theory in the space of continuous functions is
studied in Chapters 9 and 10 of~\citet{deHaan:ferreira:2000} and~\citet{einmahl:segers:2021}.
Principal component analysis of extreme curves has been studied by~\citet{kokoszka:stoev:xiong:2019},~\citet{kokoszka:kulik:2023}, and~\citet{clemenccon:huet:sabourin:2023}. Extremal properties of scores of functional data were studied by~\citet{kokoszka:xiong:2018} and Kim and Kokoszka (\citeyear{kim:kokoszka:2019},~\citeyear{kim:kokoszka:2023b}).
Additional,  more closely related papers are introduced as we
develop our approach.
We propose a method for quantifying extremal dependence of paired functional samples, for which there are currently no appropriate tools.

In the context of heavy-tailed random vectors, there has been considerable research aimed at quantifying extremal dependence. Ledford and Tawn (\citeyear{ledford:tawn:1996}, \citeyear{ledford:tawn:1997},
\citeyear{ledford:tawn:2003}) introduced the coefficient of tail dependence, which was later generalized to the extremogram by
\citet{davis:mikosch:2009}. The extremal dependence measure based on the angular measure of a regularly varying random vector was introduced by \citet{resnick:2004} and further investigated by
\citet{larsson:resnick:2012}. \citet{janssen:neblung:stoev:2023} recently introduced a unified approach for representing tail dependence using random exceedance sets. Those measures for extremes are  designed for random vectors in a Euclidean space. Therefore, applying any such measures to functional data requires some sort of dimension reduction, e.g., principal component analysis, or data compression like converting daily temperature curves to daily average or maximum values. The reduced data are then analyzed using those tools for multivariate extremes, see, e.g.,~\citet{meinguet:2010},~\citet{dombry:ribatet:2015}, and \citet{kim:kokoszka:2023b}. This approach is convenient, but it does not fully utilize all relevant information that functional data contain.

We develop a new measure, the extremal correlation coefficient, that captures the extremal dependence of paired functional samples utilizing the information in the sizes and shapes of the curves. The measure involves the inner product of pairs of extreme curves and therefore
 requires a finite second moment.   Similar ideas have been applied in non-extreme contexts of functional data analysis. \citet{dubin:muller:2005} introduced a measure, called  dynamical correlation, that computes the inner product of all pairs of standardized curves. The concept was further studied by
\citet{yeh:rice:dubin:2023} where an  autocorrelation measure, termed spherical autocorrelation, for functional time series was proposed. These measures are however computed based on the total body of functional data and so are not suitable for describing extremal dependence.

The coefficient we develop quantifies extremal dependence by specifically focusing on the extreme parts of heavy-tailed functional observations. It is conceptually appealing, as it shares desirable features with the classic correlation coefficient: (i) its values range from -1 to 1, (ii) it measures the strength and direction of linear relationship between two extreme curves, (iii) if the extremal behavior of two curves is independent, the coefficient is zero. Moreover, it can be used in practice with a relatively simple numerical implementation. We thus hope that such an interpretable and tractable tool makes a useful contribution.

Turning to mathematical challenges, the concept of vague convergence, see e.g., Chapters 2 and 6 of \citet{resnick:2007}, cannot be readily used. The vague convergence, which now provides a standard mathematical framework for extremes in Euclidean spaces, can be defined only on locally compact spaces. Since every locally compact Banach space is finite-dimensional, a different framework must be used for functional data in Hilbert spaces. We use the theory of regularly varying measures developed by~\citet{hult:lindskog:2006} who introduced the notion of $M_0$ convergence, which works for regularly varying measures on complete separable metric spaces. The $M_0$ convergence is further studied by~\citet{meinguet:2010}, where it is applied to regularly varying time series in a separable Banach space. The concept of $M_0$ convergence has been  generalized,
see Section B.1.1. of~\citet{kulik:soulier:2020}, with notable contributions from~\citet{lindskog:resnick:roy:2014} and~\citet{segers:zhao:meinguet:2017}. We establish the consistency of the estimator we propose within the framework developed by~\citet{hult:lindskog:2006} and~\citet{meinguet:2010}. We proceed through a number of $M_0$ convergence results that allow us to apply an abstract Bernstein-type inequality. A method for analytically computing the extremal correlation coefficients in relatively simple cases is also developed.



\section{Regular variation in Banach spaces} \label{s:RV}
This section presents background needed to understand the development in Sections~\ref{s:BRV} and~\ref{s:ecc}. In functional data analysis, observations are typically treated as elements of $L^2 := L^2(\cT)$, where the measure space $\cT$ is such that $L^2(\cT)$ is a {\em separable} Hilbert
space, equipped with the usual inner product $\lip x,y\rip = \int_{\cT} x(t)y(t) dt$.  The $L^2$-norm is then $\| x\| = \lip x, x\rip^{1/2}$ = $ (\int_{\cT} x(t)^2 dt)^{1/2}$. An introduction to functional data analysis is presented in~\citet{KRbook}, with a detailed mathematical treatment available in~\citet{hsing:eubank:2015}. While we refer to the elements of
$L^2$ as curves, due to the  examples we consider, the set
$\cT$ can  be a fairly abstract space (a metric Polish space),
for example a spatial domain.

An extreme curve in $L^2$ is defined as a functional object with a substantial deviation from the mean function, measured by the $L^2$-norm. The norm can be large for various reasons as long as the area under the squares of the curves around the mean function over $\cT$ is large. For example, curves that are far away from the sample mean or that fluctuate a lot around the sample mean will be extreme according to this definition. Extreme functional observations are thus very different from extreme scalar or multivariate observations because there is a multitude of ways in which a curve can be extreme. We informally call functional data heavy-tailed if the probability that an  extreme curve occurs is relatively large.

We now briefly review the $M_0$ convergence in a separable Banach space $\mbB$. In what follows, $\bzero$ is the zero element. Fix a norm $\| \cdot \|_{\mbB}$ and let $B_{\eg} := \{z \in \mbB: \| z\|_{\mbB} <\eg\}$ be the open ball of radius $\eg>0$ centered at the origin. A Borel measure $\mu$ defined on $\mbB_0:=\mbB  \setminus \{\bzero\}$ is said to be boundedly finite if $\mu(A)<\infty$, for all Borel sets that are bounded away from $\bzero$, i.e., $A \cap B_{\eg} = \emptyset$, for some $\eg>0$. Let $M_0 (\mbB)$ be the collection of all such measures on $\mbB_0$. For $\mu_n, \mu \in M_0(\mbB)$, the sequence of $\mu_n$ converges to $\mu$ in the $M_0$ topology ($\mu_n \stackrel{M_0}{\longrightarrow} \mu$),
if $\mu_n(A) \to \mu(A)$, for all bounded away from $\bzero$, $\mu$--continuity Borel sets $A$, i.e., those with $\mu(\partial A)=0$, where $\partial A$ is the boundary of $A$. Equivalently, $\mu_n \stackrel{M_0}{\longrightarrow} \mu$, if $\int_{\mbB} f(x) \mu_n(dx) \to \int_{\mbB} f(x) \mu(dx)$ for all $f \in \cC_0(\mbB)$, where $\cC_0 (\mbB)$ is the class of bounded and continuous functions $f:\mbB_0 \to \mbR$ that vanish on a neighborhood of $\bzero$.

We now define regular variation for random elements in $\mbB$, see Theorem 3.1 of~\citet{hult:lindskog:2006} and Chapter 2 of \citet{meinguet:2010}. This concept formalizes the idea of heavy-tailed observations in infinite dimensional spaces.

\begin{definition} \label{d:RVinL2}
A random element $X$ in $\mbB$ is regularly varying with index $-\ag$, $\ag> 0$, if there exist a sequence $b(n) \to \infty$ and a measure $\mu$ in $M_0(\mbB)$ such that
\begin{equation} \label{eq:def1-H}
n{\rm pr} \lp  \frac{X}{b(n)}  \in \cdot \rp
\stackrel{M_0}{\longrightarrow} \mu, \ \ \ \ n \to \infty.
\end{equation}

\end{definition}

The exponent measure $\mu$ is defined up to a multiplicative
constant and satisfies  $\mu(tA) = t^{-\ag}\mu(A)$ for any $t>0$ and Borel sets $A \subset \mbB_0$. The following lemma, see Chapter 2 of \citet{meinguet:2010}, states an equivalent definition of regularly varying  elements in $\mbB$. Throughout the paper, we denote the unit sphere in a normed space $\mbE$ by
\[
\mbS_{\mbE} := \{ u \in \mbE : \|u\|_{\mbE} = 1 \}.
\]

\begin{lemma} \label{l:pol-L2}
A random element $X$ in $\mbB$ is regularly varying with index $-\ag$, $\ag> 0$, if and only if there exist a sequence $b^{\prime}(n) \to \infty$ and a probability measure $\Gg$ on $\mbS_{\mbB}$  (called the angular measure) such that for any $z>0$,
\begin{equation} \label{eq:def2-H}
n{\rm pr} \lp \| X  \|_{\mbB} >b^{\prime}(n)z, X/\| X  \|_{\mbB} \in \cdot \rp
\stackrel{w}{\longrightarrow} c z^{-\ag}\Gg, \ \ \ \ n \to \infty,
\end{equation}
for some $c>0$.
\end{lemma}

\begin{figure}[t]
\figuresize{.5}
\figurebox{20pc}{25pc}{}[Gg]
\caption{The first three orthonormal basis elements in $L^2[0,1]$ defined in~\eqref{eq:basis} (left-most); simulated data when $\Gg$ concentrates on $\phi_1$(second from the left); on $\phi_2$ (third left); on $\phi_3$ (fourth left).
\label{f:Gg}}
\end{figure}

If Definition~\ref{d:RVinL2} (or condition~\eqref{eq:def2-H}) holds,
we write $X \in RV(-\ag, \Gg)$. The polar representation~\eqref{eq:def2-H} provides an intuitive interpretation of regular variation in $\mbB$. It characterizes regular variation of $X$ in $\mbB$ using two components, the tail index $\ag$ and the angular probability measure $\Gg$. The tail index $\ag$ quantifies the tail heaviness of the distribution of $\|X\|_{\mbB}$, with smaller values of $\ag$ implying a higher probability of extreme curves. While the tail index $\ag$ determines the frequency of occurrence of extreme curves, the angular measure $\Gg$, defined on the unit sphere $\mbS_{\mbB}$, fully characterizes the distribution of the shape of the scaled extreme curves, $X/\|X\|_{\mbB}$. To illustrate this, consider a set of orthonormal functions in $L^2([0,1])$ of the form
\begin{equation} \label{eq:basis}
\phi_j(t) = 2^{1/2} \sin \lp \lp j-\frac{1}{2}\rp \pi t\rp, \ \ \ \ j=1,2,\ldots, \ \ t \in [0,1].
\end{equation}
The first three functions are shown in the left-most plot of Figure~\ref{f:Gg}. We consider a finite-dimensional subspace of $L^2([0,1])$, spanned by the first 9 $\phi_j$'s, for the purpose of simulations. The data generating process is
$X(t) = \sum_{j=1}^{9} Z_{j} \phi_j(t)$,
where $[Z_1, \ldots, Z_9]^{\top}$ is a 9-dimensional random vector with independent components. Suppose that $Z$ is a random variable following a Pareto distribution with tail index $\ag=3$ and $N$ is a normal random variable with mean 0 and variance 0.5. We consider the following three cases for $[Z_1, \ldots, Z_9]^{\top}$:

\medskip

\noindent 1. $[Z, N, N, N, \ldots, N]^{\top}$; the angular measure $\Gg$ concentrates on $\phi_1$.

\noindent 2. $[N, Z, N, N \ldots, N]^{\top}$; the angular measure $\Gg$ concentrates on $\phi_2$.

\noindent 3. $[N, N, Z, N \ldots, N]^{\top}$; the angular measure $\Gg$ concentrates on $\phi_3$.

\medskip

In all three cases, it follows from Proposition 7.1 and Example 7.3 of~\citet{meinguet:segers:2010} that $X(t)$ is regularly varying with tail index $\ag=3$.
Fig.~\ref{f:Gg} displays simulated data with sample size of 100 for each of the three cases. The plots of simulated data clearly show that the angular measure $\Gg$ represents the distribution of the shapes of extreme curves in that they are dominated by the shape of the functional axis $\phi_j$ on which $\Gg$ concentrates.

\section{Bivariate regular variation in Banach spaces} \label{s:BRV}
In order to describe the extremal dependence of two regularly varying random elements $X$ and $Y$ in $L^2 \times L^2$, we need to identify their joint probabilistic behavior. We again study it in the more general space $\mbB^2$. We propose the following definition, which is a direct extension of Definition~\ref{d:RVinL2}.

\begin{definition} \label{d:RVinL2L2}
A bivariate random element $[X,Y]^{\top
}$ in $\mbB^2$ is said to be jointly regularly varying with index $-\ag$, $\ag> 0$, if there exist a sequence $b(n) \to \infty$ and a measure $\mu$ in $M_0(\mbB^2)$ such that
\begin{equation} \label{eq:def-RVXY}
n{\rm pr} \lp \frac{(X,Y)}{b(n)}  \in \cdot \rp
\stackrel{M_0}{\longrightarrow} \mu, \ \ \ \ n \to \infty.
\end{equation}

\end{definition}

The joint exponent measure $\mu$ satisfies the scaling property
\begin{equation} \label{eq:mu-s}
\mu(tA) = t^{-\ag}\mu(A),
\end{equation}
for any $t>0$ and Borel sets $A \subset \mbB_0^2 := \mbB^2 \setminus (\bzero, \bzero)$. Since $X$ and $Y$ are scaled by the same function $b(n)$, their marginal distributions are tail equivalent. As for a single Banach space, the measure  $\mu$ is defined
only up to a multiplicative constant, and the $b(n)$ are not unique.
To introduce suitable normalizations, consider the sets
\begin{equation} \label{eq:A}
\mathcal A_r =\{(x,y)\in \mbB^2: \|(x,y)\|_{\mbB^2} \ge r\}, \ \ \ \ r>0.
\end{equation}
According to Lemma~\ref{l:mu-conti}, the boundaries of these sets have $\mu$-measure zero. Thus, by normalizing $\mu$ using the usual condition
\begin{equation} \label{eq:mu1}
\mu(\cA_1 )=\mu\{(x,y) \in \mbB^2: \|(x,y)\|_{\mbB^2} >1\}=1,
\end{equation}
we obtain that $n {\rm pr}(\|(X,Y)\|_{\mbB^2} >b(n)) \to 1$.

Lemma \ref{l:pol-L2} implies the following equivalent characterization of regular variation in $\mbB^2$.

\begin{lemma} \label{l:BiRV}
A bivariate random element $[X,Y]^{\top}$ in $\mbB^2$ is regularly varying with index $-\ag$, $\ag>0$, if and only if there exists a sequence $b^{\prime}(n) \to \infty$ and a probability measure $\Gg$ on $\mbS_{\mbB^2}$ (called the joint angular measure) such that for any $z>0$,
\begin{equation} \label{eq:pol-RVXY}
n{\rm pr} \lp   \|(X, Y)\|_{\mbB^2}    >b^{\prime}(n)z,
\ \frac{(X,Y)}{\|(X, Y)\|_{\mbB^2}}
  \in \cdot \rp
\stackrel{w}{\longrightarrow} c z^{-\ag}\Gg, \ \ \ \ n \to \infty,
\end{equation}
for some $c>0$.
\end{lemma}

With the scaling property~\eqref{eq:mu-s} and normalization condition~\eqref{eq:mu1}, we have $c=1$ in~\eqref{eq:pol-RVXY}, and the joint angular measure $\Gg$ on $\mbS_{\mbB^2}$ can be defined by
\begin{equation} \label{eq:Gg}
\Gg(S) = \mu\{(x,y) \in \mbB^2: \|(x,y)\|_{\mbB^2} >1, (x,y)/\|(x,y)\|_{\mbB^2} \in S\},
\end{equation}
for any Borel sets $S \subset \mbS_{\mbB^2}$. Throughout the paper, we use the norm on $\mbB^2$ defined by
\[
\|(x,y)\|_{\mbB^2} := \|x\|_{\mbB} \vee \|y\|_{\mbB}.
\]
This choice works well with the extremal correlation coefficient defined in Section~\ref{s:ecc}.

Convergence~\eqref{eq:pol-RVXY} characterizes the joint extremal behavior of the bivariate regularly varying random vector $[X,Y]^{\top}$ in $\mbB^2$ in terms of the tail index $\ag$ and the joint angular measure $\Gg$ on $\mbS_{\mbB^2}$. Here, $\ag$ is the tail index of $\|X\|_{\mbB} \vee \|Y\|_{\mbB}$
and thus quantifies the frequency at which either $X$ or $Y$ exhibits an extreme size. The joint angular probability measure $\Gg$ characterizes how the shapes of normalized $X$ and $Y$ relate in extremes. If the extreme curves are exactly proportional, i.e., $X=\la Y$, $\la \neq 0$, $\Gg$ concentrates on pairs $(x,y)$ where the two components share the same functional shape up to scaling but with norms satisfying $\|x\|_{\mbB} \vee \|y\|_{\mbB}=1$. If the shapes of two curves differ substantially in extremes, $\Gg$ places mass on pairs
$(x,y)$ where the normalized curves exhibit distinct functional forms,
such as orthogonality. This situation corresponds to vanishing
extremal covariance defined in Section~\ref{s:ecc}.
The measure $\Gg$ can also describe asymptotic independence
between extremes of $X$ and $Y$, meaning that if one curve exhibits
an extreme pattern, the probability that the other curve simultaneously
exhibits an extreme pattern is negligible. In this case, $\Gg$
concentrates on pairs $(x,y)$ where either $x = 0$ or $y = 0$,
which also corresponds to vanishing extremal covariance.

The  marginal extremal behavior of $X$ can be obtained by integrating all
possible values of $Y$ in~\eqref{eq:def-RVXY}.  Then $X$ has its marginal measure $\mu_X$, and equivalently $X \in RV(-\ag, \Gg_X)$, where $\Gg_X$ is the marginal angular measure of $X$. Similarly, $Y$ has its  marginal $\mu_Y$, and equivalently $Y \in RV(-\ag, \Gg_Y)$, where $\Gg_Y$ is the marginal angular measure of $Y$. We assume that the one-dimensional marginal distributions of $\mu$ are non-degenerate, i.e., $\mu_X$ and $\mu_Y$ are measures in $M_0(\mbB)$ satisfying analogs of~\eqref{eq:def1-H}.

\section{Extremal correlation coefficient for functional data} \label{s:ecc}
In this section, we introduce the extremal correlation coefficient for functional data. It focuses on the extreme part of the joint distribution of regularly varying random vector $[X,Y]^{\top}$ in $L^2 \times L^2$ and measures the tendency of paired curves to exhibit similar extreme patterns.

To define the extremal correlation coefficient, we begin by introducing the concept of extremal covariance for functional data. Given a regularly varying bivariate random element $[X,Y]^{\top}$ in $L^2\times L^2$ with joint exponent measure $\mu$ in~\eqref{eq:def-RVXY}, we define the extremal covariance between $X$ and $Y$ by
\begin{equation} \label{eq:cov}
    \sg_{XY} = \int_{\|x\| \vee \|y\|>1} \lip x, y \rip  \mu(dx, dy).
\end{equation}
Recall that by~\eqref{eq:mu1}, $\mu$ is a probability measure on the domain $\{(x,y) \in L^2 \times L^2:\|x\| \vee \|y\|>1\}$, so $\mu I_{\|x\| \vee \|y\|>1}$ represents a probability distribution describing the  joint extremal behavior of $X$ and $Y$. The extremal covariance is thus an extreme analog of the classic covariance in that $\sg_{XY}$ measures how much two random curves vary together in extremes. In order to define the extremal covariance, $[X,Y]^{\top}$ must be regularly varying with index $-\ag$, where $\ag>2$.  The condition $\ag>2$ is necessary because the definition of extremal covariance presumes the existence of the second moment, just as the usual covariance does. We elaborate on it at the end of this
section, but note here that, as explained at the beginning of Section~\ref{s:pw},
regularly varying functions can be transformed to have the index $\ag>2$.

To interpret  the extremal covariance \eqref{eq:cov}, we focus on the
polar representation involving the joint angular measure
$\Gg$ in~\eqref{eq:Gg}. Specifically, for any
$(x,y) \neq (\bzero, \bzero)$, write $(x,y)
= r(\thg_X, \thg_Y)$, where $r:=\|(x,y)\|$ and $(\thg_X, \thg_Y) := (x,y)/\|(x,y)\|$. Then, $\mu$ can be decomposed as
\[
\mu(dx,dy) = \ag r^{-\ag-1} dr \Gg(d\thg_X,d\thg_Y).
\]
Using this, the extremal covariance $\sigma_{XY}$ can be expressed as
\begin{equation} \label{eq:cov_pol}
\sg_{XY} = \int_{r >1} r^2 \ag r^{-\ag-1} dr
\int_{\mbS_{L^2 \times L^2}} \lip \thg_X, \thg_Y \rip \Gg (d\thg_X,d\thg_Y) = \frac{\ag}{\ag-2} \int_{\mbS_{L^2 \times L^2}} \lip \thg_X, \thg_Y \rip \Gg (d\thg_X,d\thg_Y).
\end{equation}
The extremal covariance of $X$ and $Y$ can thus be decomposed into
the tail index factor $\ag/(\ag - 2)$ and the angular dependence
given by the integral.

\begin{table}[t]
\caption{Ranges of $\sg_{XY}$, the extremal covariance between $X$ and $Y$. Asymptotic independence refers to a joint distribution where
extreme $X$ and $Y$ rarely occur simultaneously.}  \label{t:ga}
\centering
\medskip
\begin{tabular}{c|lc}
\hline\hline
Asymptotic Independence & \multicolumn{2}{c}{Asymptotic Dependence} \\
\hline
& $X/(\|X\| \vee \|Y\|)$ and $Y/(\|X\| \vee \|Y\|)$ & \\
$\sg_{XY} \approx 0$& \ \ look similar & $\sg_{XY} >0$  \\
& \ \ look orthogonal & $\sg_{XY}\approx 0$ \\
& \ \ look opposite & $\sg_{XY}<0$ \\
\hline \hline
\end{tabular}
\end{table}

While the tail index factor represents the occurrence
intensity of either extreme $\| X\|$ or $\| Y\|$, the angular
dependence part captures the mutual relationship between the curves and how the shapes of the normalized extreme curves co-vary. More precisely, if the extremes of $X$ and $Y$ are asymptotically independent, meaning that extreme curves in $X$ and $Y$ rarely occur simultaneously, then $\Gamma$ places most mass near pairs where one component is close to the zero element, causing the integral to be close to zero, and thus $\sigma_{XY}$ will be near zero regardless of the shapes of the extreme curves. If $X$ and $Y$ are asymptotically dependent, meaning that extreme curves tend to occur simultaneously in $X$ and $Y$, then the value of $\sigma_{XY}$ depends on the shape similarity between the normalized curves $X/(\|X\| \vee \|Y\|)$ and $Y/(\|X\| \vee \|Y\|)$. Consequently, there are three possible ranges for $\sg_{XY}$, depending on the relative shape of these normalized extremes: (i) $\sg_{XY}>0$ when the shapes are similar, (ii) $\sg_{XY} \approx 0$ when the shapes do not match and are approximately orthogonal, or (iii) $\sg_{XY}<0$ when the shapes are opposite. These properties are summarized in Table~\ref{t:ga}.

To provide a scale-invariant measure of extremal dependence, we next define the extremal correlation coefficient.

\begin{definition} \label{d:cor}
The extremal correlation coefficient is defined by
\begin{equation} \label{eq:cor}
\rho_{XY} = \frac{\sg_{XY}}{\sg_{X}\sg_{Y}},
\end{equation}
where the extremal covariance $\sg_{XY}$ is defined by \eqref{eq:cov}, and
where
\begin{align*}
\sg_{X} = \lp \int_{\|x\| \vee \|y\|>1}  \|x\|^2   \mu(dx,dy) \rp^{1/2}, \ \ \ \
\sg_{Y} = \lp \int_{\|x\| \vee \|y\|>1}  \|y\|^2   \mu(dx,dy) \rp^{1/2}.
\end{align*}
\end{definition}
The coefficient $\rho_{XY}$ has  properties  analogous to the
classic correlation coefficient: (i) $-1\le \rho_{XY} \le 1$, (ii) $\rho_{XY}$
measures the strength and direction of linear relationships between $X$
and $Y$ in extremes, (iii) if $X$ and $Y$ are independent,
then $\rho_{XY}=0$ since independence implies asymptotic
independence between $\|X\|$ and $\|Y\|$. We note that
$\rho_{XY}$ may depend on $\ag$ if $\Gg$ depends on $\ag$.
An example is provided in Lemma~\ref{l:sim1}, see also Example 2.2.8 in
\citet{kulik:soulier:2020}. If $\Gg$ does not depend on
$\ag$, then $\rho_{XY}$ does not depend on $\ag$.

To motivate our estimation approach,
we first show that  $\sg_{XY}$ is a  limit of the expected inner product of scaled $X$ and $Y$ conditional on large values of $[X,Y]^{\top}$.

\begin{proposition} \label{p:sglim}
Let $[X,Y]^{\top}$ be a regularly varying random element in $L^2 \times L^2$. Then,
\[
\sg_{XY} = \lim_{n\to \infty}E\lb \lip \frac{X}{b(n)}, \frac{Y}{b(n)} \rip \Bigg|\|X\| \vee \|Y\| >b(n)\rb.
\]
\end{proposition}

\begin{proof}
Considering $f: L^2 \times L^2 \to \mbR$ defined by $(x, y) \to \lip x, y \rip I_{\|x\| \vee \|y\|>1}$, we have that
\begin{align*}
&E\lb \lip {b(n)}^{-1}X, {b(n)}^{-1}Y \rip | \|X\| \vee \|Y\| >{b(n)}\rb \\
&=\frac{1}{{\rm pr}(\|X\|\vee\|Y\|>{b(n)})} E\lb \lip {b(n)}^{-1}X, {b(n)}^{-1}Y \rip I_{\|X\| \vee \|Y\| >{b(n)}}\rb\\
&=\frac{1}{n{\rm pr}(\|X\|\vee\|Y\|>{b(n)})}\int_{L^2 \times L^2} f(x,y) \ n{\rm pr}({b(n)}^{-1}X \in dx, {b(n)}^{-1}Y \in dy).
\end{align*}
By~\eqref{eq:mu1}, we have that $n{\rm pr}(\|X\|\vee\|Y\|>{b(n)}) \to 1$.
The function $f$ vanishes on a neighborhood of
$(\bzero, \bzero)$ in $L^2 \times L^2$. Also, the discontinuity set of $f$ is the
boundary of $\mathcal A_1 = \{(x,y) \in L^2 \times L^2: \|x\| \vee \|y\| \ge 1\}$,
and it follows from Lemma~\ref{l:mu-conti}
that $\mu(\partial \mathcal A_1)=0$.
Therefore, by~\eqref{eq:def-RVXY} and an extension of Lemma A.1 
of~\citet{meinguet:segers:2010} (see Lemma~\ref{l:cont_m0}),  
we get the claim. 
\end{proof}

Based on Proposition~\ref{p:sglim}, we propose an estimator for $\sg_{XY}$ defined by
\begin{equation} \label{eq:ec}
\hat{\sg}_{n,k} =\frac{1}{k}\sum_{i=1}^n \lip \frac{X_i}{R_{(k)}}, \frac{Y_i}{R_{(k)}} \rip I_{R_i \ge R_{(k)} },
\end{equation}
where $[X_i, Y_i]^{\top}, 1 \le i \le n$, are i.i.d. copies of $[X,Y]^{\top}$,   $R_i:=\|X_i\| \vee \|Y_i\|$ and $R_{(k)}$ is the $k$th largest order statistic with the convention $R_{(1)} = \max \{ R_1, \ldots, R_n \}$. An estimator for $\rho_{XY}$ is then defined by
\begin{equation} \label{eq:ecc}
\hat{\rho}_{n,k} = \frac{\sum_{i=1}^n \lip X_i, Y_i \rip}{(\sum_{i=1}^n \|X_i\|^2)^{1/2} (\sum_{i=1}^n \|Y_i\|^2)^{1/2}} \ I_{R_i \ge R_{(k)} }.
\end{equation}

These estimators take only the $k$ largest pairs of $[X_i, Y_i]^{\top}$, $1\le i\le n$, based on their norm, i.e., $\|X_i\| \vee \|Y_i\|$, as inputs. This approach falls into so-called peaks-over-threshold framework in that it relies only on $k$ largest observations whose magnitude exceeds a certain threshold. Asymptotic properties in this framework are typically derived as $k$ goes to infinity with $n$, in such a way that $k/n \to 0$. We assume throughout the paper that this condition holds.

We will work under the following assumption.

\begin{assumption} \label{a:L2L2}
The bivariate random element $[X,Y]^{\top}$ in  $L^2 \times L^2$ has mean zero and is regularly varying with index $-\ag$, $\ag >2$. The observations $[X_1,Y_1]^{\top}, [X_2,Y_2]^{\top}, \ldots$ are independent copies of $[X,Y]^{\top}$.
\end{assumption}

We state in the following theorem that the estimator $\hat{\sg}_{n,k}$ is consistent for the extremal covariance. All proofs of the theoretical results introduced in this section are presented in Sections~\ref{s:pre} and~\ref{s:pr} of the Supplementary material, as they require a number of preliminary  results and technical arguments.

\begin{theorem}\label{t:cs}
Under Assumption~\ref{a:L2L2},
$
\hat{\sg}_{n,k} \convP \sg_{XY}$,
where $\hat{\sg}_{n,k}$ and $\sg_{XY}$ are defined in \eqref{eq:ec} and \eqref{eq:cov}, respectively.
\end{theorem}

The consistency of $\hat{\rho}_{n,k}$ for $\rho_{XY}$ follows from  Theorem~\ref{t:cs} and Slutsky's theorem.

\begin{corollary}\label{c:cs}
Under Assumption~\ref{a:L2L2},
$\hat{\rho}_{n,k} \convP \rho_{XY}$,
where $\hat{\rho}_{n,k}$ and $\rho_{XY}$  are defined in \eqref{eq:ecc} and
\eqref{eq:cor}, respectively.
\end{corollary}

We end this section with a discussion on the condition $\ag>2$ in Assumption~\ref{a:L2L2}. This requirement ensures the existence of the second moments, $E \|X\|^2$ and $E \|Y\|^2$, allowing us to define $\sg_{XY}$. If one wishes to avoid the condition on $\ag$, the following alternative measure can be considered:
\[
\ga_{XY} := \int_{\mbS_{L^2 \times L^2}} \lip \thg_X, \thg_Y \rip \Gg(d\thg_X, d\thg_Y),
\]
which corresponds to the angular density factor in~\eqref{eq:cov_pol}.
This measure depends only on the angular measure $\Gg$,
and is therefore defined for any $\ag>0$. It can be regarded as
an extension of the {\it extremal dependence measure}
of~\citet{larsson:resnick:2012} to paired functional samples,
since their measure is also based on an angular measure.
An estimator for $\ga_{XY}$ can be defined as
$\hat{\ga}_{n,k} = \frac{1}{k} \sum_{i=1}^n
\lip X_i/R_i, Y_i/R_i \rip I_{R_i \ge R_{(k)}}$,
and its consistency can be proven in almost the same manner
as  the proof of Corollary 4.2 of~\citet{clemenccon:huet:sabourin:2023}.
The measure $\gamma_{XY}$ provides a convenient way
to quantify the similarity in the shapes of extreme curves without requiring $\ag>2$.
However, $\gamma_{XY}$ does not account for the frequency of occurrence
of extreme $X$ or $ Y$ as $\sg_{X, Y}$ does.

\section{A simulation study} \label{s:pw}
In this simulation study, we demonstrate that the proposed estimator
$\hat{\rho}_{n,k}$ consistently estimates the extremal correlation coefficient.
Before proceeding, we explain how $\hat{\rho}_{n,k}$ can be computed when curves are observed at discrete points. This will be applied throughout Sections~\ref{s:pw} and~\ref{s:dt}.

Assume that curves are observed on the regularly spaced grid $\{j/J, \ j \in \{1, \ldots, J\} \}$ on $[0,1]$, with each point assigned equal weight. The inner product and norm in $L^2([0,1])$ are then
\begin{equation} \label{eq:inner}
\lip x,y\rip = \frac{1}{J} \sum_{j=1}^Jx(j/J)y(j/J), \quad \|x\| = \lb \frac{1}{J} \sum_{j=1}^Jx(j/J)^2 \rb^{1/2}, \quad x,y \in L^2([0,1]).
\end{equation}
Even if the curves are observed at irregularly spaced or different  grids, or contain missing values, they can be  reconstructed on a regularly spaced grid, see e.g., Chapters 1 and 7 in \citet{KRbook}.
Using~\eqref{eq:inner}, $\hat{\rho}_{n,k}$ can be computed in  the following steps:

\medskip

\noindent Step 1. Verify if $\|X\|$ and $\|Y\|$ are regularly varying, for example, by examining whether their  Hill plots  exhibit stable regions.

\medskip

\noindent Step 2. Estimate the tail indexes of $\|X\|$ and $\|Y\|$  using the Hill estimator. For this, we use the {\tt mindist} function from the {\tt R} package {\tt tea}.

\medskip

\noindent Step 3. If the tail index estimates from Step 2 are not close to each other, apply a transformation to make $X$ and $Y$ tail equivalent. One approach is the power transformation discussed on page 310 of~\citet{resnick:2007}. Given $X \in RV(-\ag_X, \Gg_X)$ and $Y \in RV(-\ag_Y, \Gg_Y)$, consider the transformation
\begin{equation} \label{eq:tr}
g_X(x) = \frac{x}{\|x\|^{1-{\ag_X}/\ag}}; \ \ \ \ g_Y(y) = \frac{y}{\|y\|^{1-{\ag_Y}/\ag}},\ \ \ \ x,y \in L^2,
\end{equation}
where $\ag$ is a desired tail index. Applying $g_X$ and $g_Y$ to $X$ and $Y$,
respectively, ensures that $P(\|g_X(X)\|>\cdot)$ and $P(\|g_Y(Y)\|>\cdot)$
are regularly varying with the same index $-\ag$. Since this method adjusts
only the scale of curves, their shapes are preserved.

\medskip

\noindent Step 4. Given tail equivalent marginals, take the $k$ largest pairs of $[X_i, Y_i]^{\top}$, based on their norm, i.e., $R_i = \|X_i\| \vee \|Y_i\|$, and then compute $\hat{\rho}_{n,k}$ as in~\eqref{eq:ecc}.  By Lemma~\ref{l:R} (i), if $[X, Y]^{\top}$ are regularly varying in $L^2 \times L^2$ with index $-\alpha$, then $R = \|X\| \vee \|Y\|$ is regularly varying in $\mathbb{R}_+$ with the same index. Therefore, we choose $k$ that results in successful tail estimation for $R$ in finite samples. In the literature on tail estimation, various methods for selecting $k$ have been introduced, e.g.,~\citet{hall:Welsh:1985},~\citet{hall:1990},~\citet{drees:kaufmann:1998},~\citet{danielsson:deHaan:Peng:deVries:2001}, just to name a few.  We use  the method of~\citet{danielsson:Ergun:deHaan:deVries:2016}, which  is implemented using the function {\tt mindist} of the {\tt R} package {\tt tea}.

\medskip

We now outline the design of our simulation study. We generate functional observations in such a way that the theoretical value of $\rho_{XY}$ can be computed analytically, so that we can see how close $\hat{\rho}_{n,k}$ is to the true value. Suppose that $Z_1$ and $Z_2$ are i.i.d. random variables in $\mbR$ satisfying ${\rm pr}(|Z_{1}|>z) = z^{-\ag}$ with equal chance of $Z_1$ being either negative or positive. Also, let $N_1$, $N_2$, and $N_3$ be i.i.d. normal random variables in $\mbR$ with mean 0 and variance 0.5. Consider $\{\phi_j, j \ge 1 \}$ defined by~\eqref{eq:basis} and recall that it is an  orthonormal basis in $L^2([0,1])$. These functions are simulated on a grid of 100 equally--spaced points on the unit interval $[0,1]$. We consider the following data generating processes, for $-1 \le \rho \le 1$,
\begin{align} \label{eq:XY}
X(t) &= Z_{1}\phi_1(t) + N_1\phi_2(t) +  N_2\phi_3(t); \\ \nonumber
Y(t) &= \rho Z_{1}\phi_1(t) + (1-\rho^2)^{1/2} Z_{2}\phi_2(t) + N_3\phi_3(t).
\end{align}
It generates extreme curves dominated by the shape of the functional axis $\phi_1$ for $X$ and  by either $\phi_1$ or $\phi_2$ for $Y$.  The following lemma gives an analytic formula for $\rho_{XY}$. Its proof (and a slightly more general result) is provided in Section~\ref{s:pr_sim} of Supplementary material.

\begin{lemma} \label{l:sim1} Let $[Z_1, Z_2]^{\top}$ be a random vector in $\mbR^2$ consisting of i.i.d. components $Z_j$ whose magnitude is regularly varying with $-\ag$, $\ag>2$, i.e., for some $c_{+}$, $c_{-}\ge 0$,
\[
{\rm pr}(Z_1 >z) \sim c_{+}z^{-\ag}, \quad\quad {\rm pr}(Z_1 <-z) \sim c_{-}z^{-\ag},
\]
where $f(z) \sim g(z)$ if and only if $\lim_{z\to \infty}f(z)/g(z)=1$.
Also, let $\{\phi_j,\ j \ge 1\}$ be a set of orthonormal elements in $\mbS_{L^2}$. Then, for  $X$ and $Y$ in \eqref{eq:XY},
\[
\rho_{XY} = \frac{\rho}{\{\rho^2 + (1-\rho^2)^{\ag/2}\}^{1/2}}.
\]
\end{lemma}

We consider $\rho_{XY} \in \{0, \pm0.1, \pm0.2, \ldots, \pm 0.9, \pm 1 \}$ and $\ag \in\{3,4,5\}$, from which values of $\rho$ can be obtained by Lemma~\ref{l:sim1}. For each $\rho$, we generate $[X_i, Y_i]^{\top}$, $1 \le i \le n$, that are i.i.d. copies of $[X,Y]^{\top}$, with sample sizes $n \in \{100, 500, 2000\}$. In each case, 1000 replications are generated. By Proposition 7.1 and Example 7.3 of~\citet{meinguet:segers:2010}, $X$ and $Y$ in~\eqref{eq:XY} are regularly varying with the same index $-\ag$, so we proceed directly to Step 4 and compute $\hat{\rho}_{n,k}$. When choosing $k$, we also consider an alternative approach based on the Kolmogorov–Smirnov (KS) distance, as introduced by~\citet{clauset:shalizi:newman:2009}. The key difference from~\citet{danielsson:Ergun:deHaan:deVries:2016} is that~\citet{clauset:shalizi:newman:2009} computes the distance from tail distributions, rather than tail quantiles. The method is implemented using {\tt poweRlaw} {\tt R} package.

\begin{table}[t]
\centering
\caption{The magnitude of empirical biases (standard errors) of $\hat{\rho}_{n,k}$ when $\ag=3$. Optimal $k$s are selected using the method from~\citet{danielsson:Ergun:deHaan:deVries:2016}, with averages of $k=8$ ($N=100$), $k=26$ ($N=500$), and  $k=63$ ($N=2000$).
\label{t:sim2}}
\medskip
\begin{tabular}{cccc}
  \hline\hline
 $\rho_{XY}$ & $N=100$ &  $N=500$ & $N=2000$  \\
  \hline
-1.0 & 0.04 (0.03) & 0.02 (0.01) & 0.01 (0.01) \\
-0.9 & 0.06 (0.13) & 0.04 (0.09) & 0.03 (0.09) \\
-0.8 & 0.07 (0.16) & 0.04 (0.13) & 0.03 (0.12) \\
-0.7 & 0.04 (0.18) & 0.02 (0.15) & 0.01 (0.14) \\
-0.6 & 0.03 (0.19) & 0.02 (0.16) & 0.01 (0.14) \\
-0.5 & 0.01 (0.19) & 0.00 (0.16) & 0.01 (0.15) \\
-0.4 & 0.01 (0.18) & 0.00 (0.15) & 0.01 (0.15) \\
-0.3 & 0.01 (0.16) & 0.01 (0.13) & 0.02 (0.14) \\
-0.2 & 0.02 (0.14) & 0.02 (0.12) & 0.01 (0.11) \\
-0.1 & 0.01 (0.09) & 0.01 (0.08) & 0.01 (0.09) \\
0.0 & 0.00 (0.06) & 0.00 (0.04) & 0.00 (0.03) \\
0.1 & 0.01 (0.10) & 0.00 (0.08) & 0.01 (0.09) \\
0.2 & 0.01 (0.14) & 0.02 (0.11) & 0.02 (0.11) \\
0.3 & 0.02 (0.16) & 0.02 (0.15) & 0.01 (0.13) \\
0.4 & 0.00 (0.18) & 0.02 (0.15) & 0.01 (0.15) \\
0.5 & 0.00 (0.19) & 0.00 (0.16) & 0.00 (0.15) \\
0.6 & 0.04 (0.18) & 0.02 (0.16) & 0.00 (0.15) \\
0.7 & 0.04 (0.17) & 0.03 (0.15) & 0.01 (0.14) \\
0.8 & 0.06 (0.16) & 0.03 (0.12) & 0.03 (0.13) \\
0.9 & 0.06 (0.12) & 0.04 (0.10) & 0.03 (0.09) \\
1.0 & 0.04 (0.03) & 0.02 (0.01) & 0.01 (0.01) \\
\hline\hline
\end{tabular}
\end{table}

We report the magnitude of empirical biases, measured by the absolute difference between the average and the theoretical value, along with standard errors computed as the sample standard deviations.
Using the optimal $k$s selected by the method from~\citet{danielsson:Ergun:deHaan:deVries:2016}, the results are shown in Table~\ref{t:sim2} for $\alpha = 3$. The results for $\ag \in \{4,5\}$, as well as those using the method from \citet{clauset:shalizi:newman:2009}, are provided in Section~\ref{ss:a_sim} of Supplementary material, as they lead to similar conclusions. The conclusions are summarized as follows.

\noindent (i) The estimator is consistent as the biases decrease with increasing sample sizes, across nearly all values of $\rho_{XY}$ and for all $\ag \in \{3,4,5\}$ considered. However, the finite performance using the optimal $k$s appears to depend on the tail index $\ag$, with smaller $\ag$ resulting in small biases.

\noindent (ii) The bias tends to increase in magnitude as $|\rho_{XY}|$ approaches 1. This could be due to the effect of the boundary, $\rho_{XY} \in \{-1, 1\}$. These barriers cause the bias to increase, although it gets lower again at the boundary.

\noindent (iii) The standard errors are observed  to be non-uniform across $\rho_{XY}$, they roughly behave like a quadratic function of $\rho_{XY}$ with its peak at $\pm.5$.

Additional information about
$\hat\rho_{n,k}$ is provided in the online Supplementary material.
Specifically, Section
\ref{ss:a_sim} examines the effect of the
selection of $k$, Section \ref{ss:a_chi} explores the relation to the extremal measures $\chi$ and $\bar{\chi}$ introduced by \citet{coles:heffernan:tawn:1999}, and Section \ref{ss:a_phv} investigates the effect of phase variation.

\section{Applications to financial and climate functional data} \label{s:dt}

\subsection{Extremal dependence of  intraday returns
on sector ETFs} \label{ss:etf}

In this section, we study pairwise extremal dependence of cumulative intraday return curves (CIDRs) of  Exchange Traded Funds (ETFs) reflecting
performance of key sectors of the U.S. economy. We work with  nine
Standard \& Poor's  Depositary Receipt ETFs listed in Table~\ref{t:etf}.
Our objective is to
measure the tendency of paired CIDRs to exhibit similar extreme daily trajectories during the market decline caused by the covid-19 pandemic.
The CIDRs are defined as follows. Denote by $P_i(t)$ the price of an asset
on trading day $i$ at time $t$. For the assets in our example, $t$ is time in minutes between 9:30 and 16:00 EST (NYSE opening times) rescaled to the unit interval $(0,1)$. The CIDR on day $i$ is  the curve $R_i(t) = \ln P_i(t) - \ln P_i(0)$, $t\in [0,1]$, where $P_i(0)$ is the opening price on day $i$. The curves $R_i$ show how the return accumulates over the trading day, see  Figure~\ref{f:pairs}. We consider all full trading days between  Jan 02, 2020 and  July 31, 2020 ($N=147$).

\begin{table}[t]
\centering
\caption{The nine sector ETFs and their corresponding tail index estimates $\hat{\alpha}$.}
\label{t:etf}
\medskip
\begin{tabular}{llcllc}
  \hline\hline
Ticker & Sector & $\hat{\alpha}$ & Ticker & Sector & $\hat{\alpha}$ \\
\hline
XLY & Consumer & 3.8 & XLV & Health Care & 3.9 \\
    & Discretionary && XLI & Industrials & 3.7 \\
XLP & Consumer Staples & 2.6 & XLB & Materials & 3.4 \\
XLE & Energy  & 4.2 & XLK & Technology  & 4.7 \\
XLF & Financials & 4.0 & XLU & Utilities  & 3.8 \\
\hline\hline
\end{tabular}
\end{table}

\begin{figure}[t]
\begin{centering}
\includegraphics[width=1\textwidth]{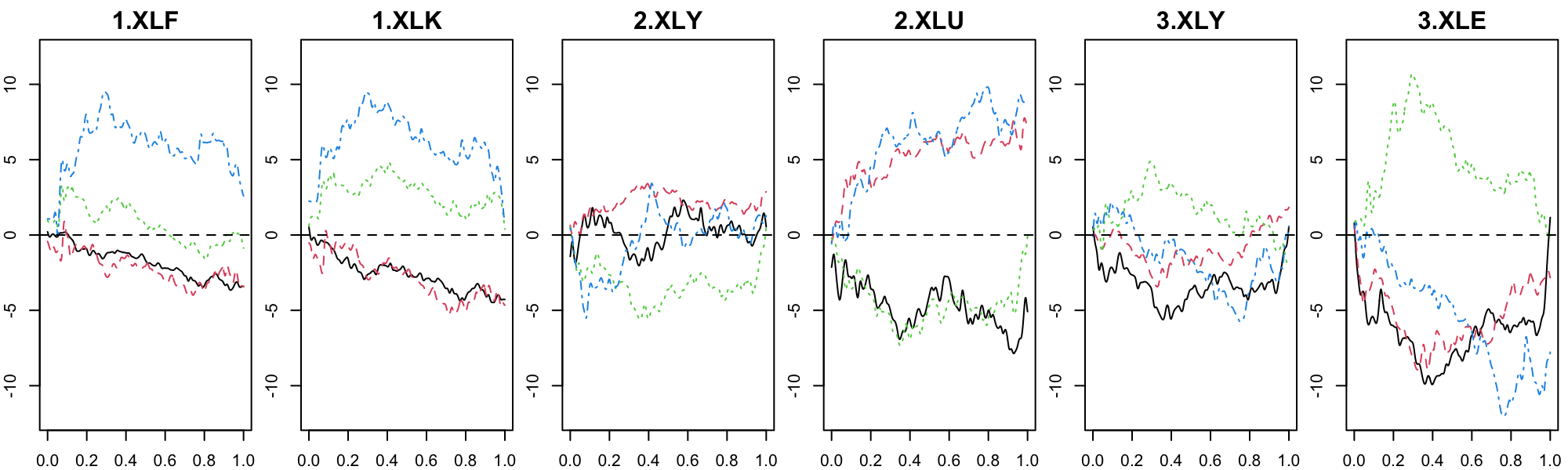}
\par\end{centering}
\caption{The CIDR of three pairs of ETFs (1.XLF and XLK, 2.XLY and XLU, 3.XLY and XLE). For each pair, the curves representing the four most extreme days are displayed, with matching colors and line types indicating curves from the same day.
\label{f:pairs}}
\end{figure}

We follow the step-by-step guide for estimating $\rho_{XY}$ presented in Section~\ref{s:pw}. First, for each sector, we center the curves around their sample mean functions, $\bar{R}_N (t) = \frac{1}{N}\sum_{i=1}^N R_i(t)$, and compute its norm $\|R_i(t) -\bar{R}_N (t)\|$ using~\eqref{eq:inner} with $J = 390$. We then examine whether the Hill plots of the norms for each sector exhibit stable regions. As shown in Fig~\ref{f:Hill_e} of Supplementary material, the norms appear regularly varying, so we compute the tail estimates $\hat\ag$ for each sector using the Hill estimator, as shown in Table~\ref{t:etf}. Since the sectors are not tail equivalent, we apply the power transformation~\eqref{eq:tr} to achieve tail equivalence with $\ag=3$. This choice yields small biases in finite samples, as shown in Section~\ref{s:pw}. Next, we use the methods from~\citet{danielsson:Ergun:deHaan:deVries:2016} to determine the optimal $k$ for estimating $\rho_{XY}$ for each pair across the sectors.

\begin{figure}[t]
\begin{centering}
\includegraphics[width=.65\textwidth]{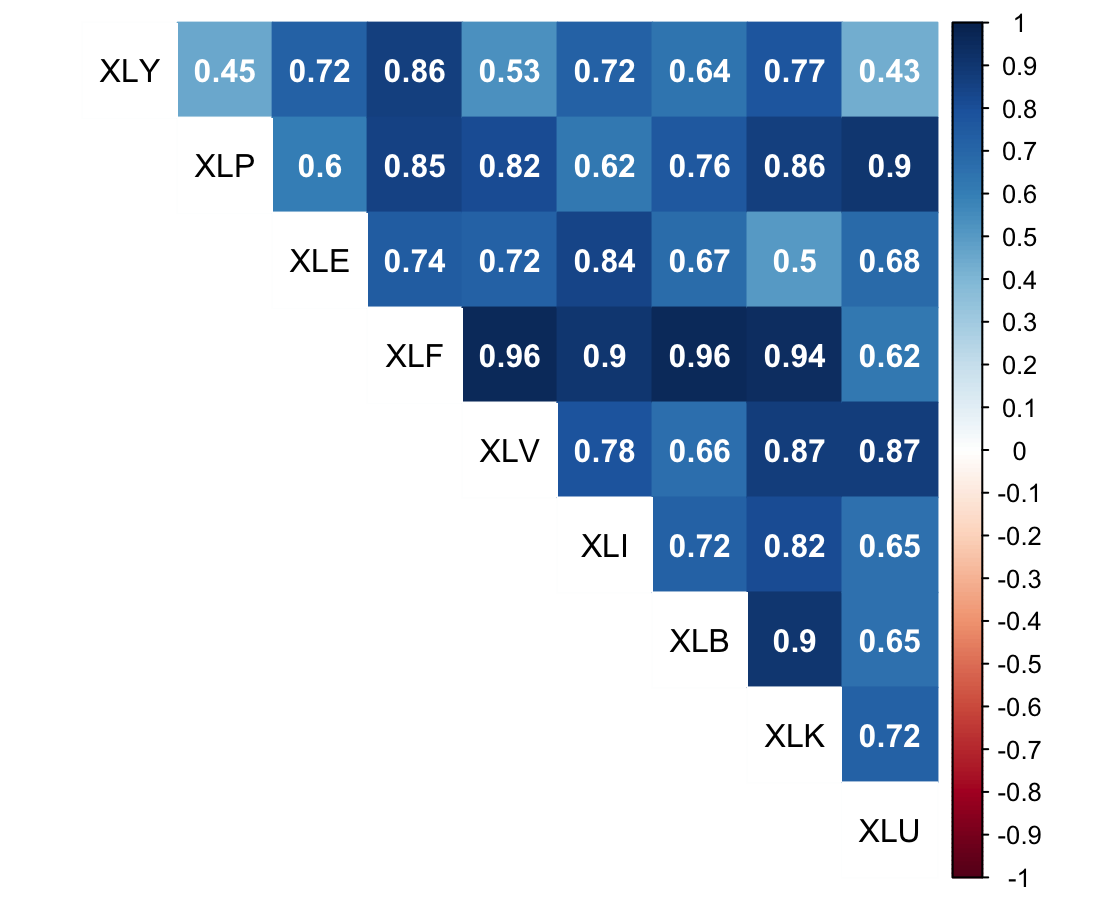}
\par\end{centering}
\caption{Estimates of the pairwise extremal correlation coefficients of CIDRs across the nine sectors.
\label{f:ecc_cidr}}
\end{figure}

Figure~\ref{f:ecc_cidr} shows estimates of the pairwise extremal correlation coefficient across the nine ETF sectors. All pairs exhibit positive extremal correlations ($\hat{\rho}_{n,k} = 0.43 \sim 0.96$), and 44\% of the pairs have strong extremal correlations above 0.7. We see that the CIDRs overall exhibit matching  patterns of cumulative
intraday returns on extreme market volatility days during the covid-19 market turbulence, where most sectors either drop or increase together. However, our coefficient reveals more subtle
information as well. For example,  extreme return curves of
XLF (Financials) are exceptionally strongly correlated with extreme curves for
XLV, XLB and XLK (Health Care,  Materials, Technology), but moderately correlated with XLU (Utilities). We do not aim at an analysis of the stock market or the economy, but we observe that some findings are interesting. One might expect the financial sector (mostly banks) to be strongly affected by the technology sector (mostly large IT companies like Google or Microsoft) because such mega corporations dominate the U.S. stock market. The similarity of extreme return curves for XLF and XLK is shown in the leftmost panels of Fig.~\ref{f:pairs}. In contrast, bank stocks would be expected to be less affected by utility companies, whose revenues are largely fixed. However, the strong correlation with the Health Care and Materials sectors is less clear. As another comparison, XLY (Consumer Discretionary) and XLU (Utilities) show a moderate extremal correlation of 0.43. Their extreme curves exhibit relatively dissimilar patterns, as seen in the middle panels in Fig.~\ref{f:pairs}.

\begin{figure}[t]
\begin{centering}
\includegraphics[width=\textwidth]{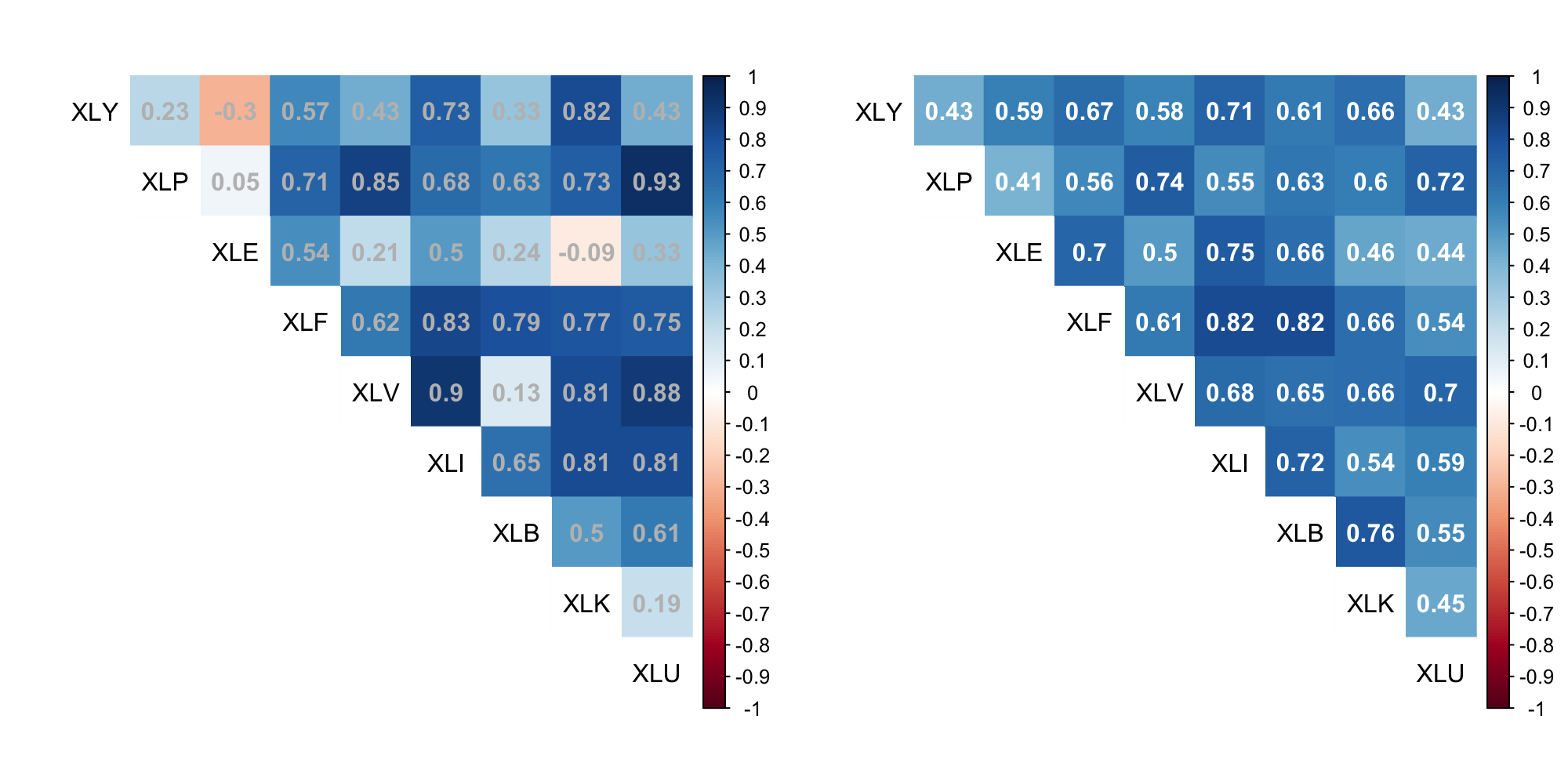}
\par\end{centering}
\caption{Estimates of the pairwise coefficients of CIDRs, calculated from closing returns (left) and from all curves including non-extreme parts (right), are displayed.
\label{f:ecc_cidr_c}}
\end{figure}

We conclude by emphasizing that our tool offers a more precise quantification of intraday risk during extreme events. First, by analyzing curve shapes, it provides a better assessment of intraday risks. The left plot of Figure~\ref{f:ecc_cidr_c} displays the pairwise coefficients from single-valued closing returns. This plot reveals somewhat different information from Fig.~\ref{f:ecc_cidr}. For instance, between XLY and XLE, the closing returns show a $-0.3$ correlation, indicating negative correlation during extreme days. However, this value does not accurately reflect the positive relationship observed in the rightmost plots in Fig.~\ref{f:pairs}, where the paired extreme curves appear somewhat similar. Second, by focusing on the extreme parts of paired samples of curves, our tool effectively quantifies risk during these critical events. The right plot of Fig.~\ref{f:ecc_cidr_c} shows  correlation coefficients computed from all curves, including non-extreme ones. These coefficients tend to underestimate the extreme risk, highlighting the necessity for tools that specifically describe extreme conditions.

\subsection{Extremal correlation between daily temperature curves} \label{ss:hw}

\begin{figure}[t]
\begin{centering}
\includegraphics[width=.6\textwidth]{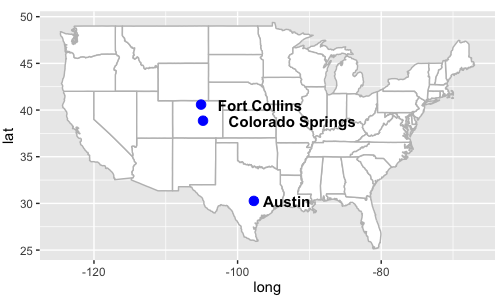}
\par\end{centering}
\caption{The three locations in the United States: Fort Collins, CO; Colorado Springs, CO; Austin, TX. The pairwise extremal correlation of daily temperature curves between the three locations is evaluated.
\label{f:map}}
\end{figure}

In this section, we evaluate the tendency of paired daily temperature curves to exhibit similar extreme patterns across three locations in the United States. The three locations are marked  in Figure~\ref{f:map}. We focus on the pairwise extremal dependence of those curves during the 2021 heat wave. Although this example  focuses on temperature curves, our tool can be used
for analyzing other curves during extreme weather events;   for example,
daily precipitation patterns or river flows during floods. A correlation of extreme data   during   past events may help with planning a resilient
infrastructure that can better withstand the next extreme weather event.

We use hourly temperature measurements provided by the European Centre for Medium-Range Weather Forecasts (ECMWF). The data are part of their ERA5 (Fifth Generation of ECMWF atmospheric reanalyses) dataset, and represent the temperatures of air at 2 meters above the surface of land, sea or inland waters. We refer to~\citet{era5:2020} for more details on the ERA5 data.  We partition the hourly data into daily curves, with each day's curve starting at midnight local time, to produce comparable daily temperature curves across locations in different time zones. We denote the temperature (in Celsius) on day $i$ at hour $t$ by $X_i(t)$, $i=1, \ldots, N$.
Figure~\ref{f:3loc} depicts examples of daily temperature curves at the three locations. The data are taken from May 12, 2021 to Aug 31, 2021 ($N$ = 112).

\begin{figure}[t]
\begin{centering}
\includegraphics[width=1\textwidth]{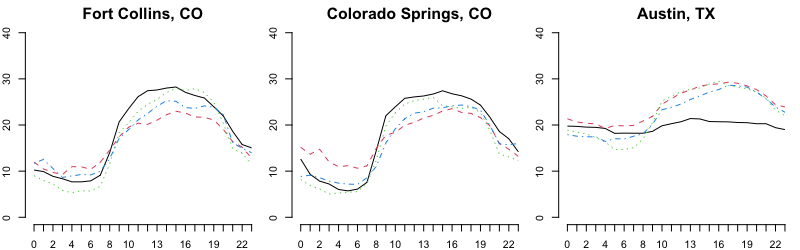}
\par\end{centering}
\caption{Extreme daily temperature curves (in Celsius) during the 2021 heat wave (local time on the x-axis).  Curves of matching color represent the same days when both Fort Collins and Colorado Springs   experienced extreme patterns simultaneously.
\label{f:3loc}}
\end{figure}

We follow the step-by-step guide outlined in Section~\ref{s:pw} to compute $\hat{\rho}_{n,k}$ for each pair of the three locations. First, for each location, daily curves are centered by the mean function, $\bar{X}_N (t) = \frac{1}{N}\sum_{i=1}^N X_i(t)$, and its norm $\| X_i(t) -\bar{X}_N (t)\|$ is computed using~\eqref{eq:inner} with $J = 24$. We then examine whether the Hill plots of the norms for each location show stable regions. As seen in Figure~\ref{f:Hill_t} of Supplementary material, the norms are regularly varying, so we compute the tail estimates $\hat{\alpha}$ using the Hill estimator, as shown in Table~\ref{t:etf}. Since the marginals across the three locations are not tail-equivalent, we apply the power transformation~\eqref{eq:tr} to achieve tail equivalence with $\alpha = 3$. We then apply the method from~\citet{danielsson:Ergun:deHaan:deVries:2016} to determine the optimal $k$ for estimating $\rho_{XY}$ for each pair.

\begin{table}[t]
\centering
\caption{Tail index estimates $\hat{\ag}$ and pairwise extremal correlation coefficients $\hat{\rho}_{n,k}$ of daily temperature curves across Fort Collins, CO, Colorado Springs, CO, and Austin, TX.
\label{t:3loc}}
\medskip
\begin{tabular}{l|c|cccc}
  \hline\hline
 &  && & $\hat{\rho}_{n,k}$ & \\ \cline{4-6}
Location & $\hat{\ag}$ && Fort & Colorado  & Austin\\
 &  && Collins & Springs & \\
\hline
Fort Collins & 4.4 && 1& 0.98 & 0.83\\
Colorado Springs & 3.8 &&0.98 & 1& 0.86\\
Austin& 3.4 &&0.83 & 0.86 &1\\
\hline\hline
\end{tabular}
\end{table}

Table~\ref{t:3loc} reports estimates of the pairwise  extremal correlation coefficient across the three locations. There are positive and strong extremal correlations among all pairs ($\hat{\rho}_{n,k} = 0.83 \sim 0.98$), suggesting a high degree of association between the daily temperature extreme patterns across the three locations, even between different
climatic regions like the Front Range  foothills and the southern edge of the Great Plains.
We see that the proximity in geographical locations corresponds to greater similarity in extreme patterns, showing that  $\hat{\rho}_{n,k}$ is
a meaningful and useful dependence measure.

\section{Discussion}

Despite  promising results, there are some limitation to our approach that suggest potential directions for future work. We treat functional
observations as regularly varying square integrable
random functions in  $L^2$, which requires  $\ag>2$. While this condition is met for  the financial and environmental data we work with, it might be  desirable to find a correlation-like extremal dependence measure that requires merely $\ag>0$.
Possible extensions could involve the codifference or  the covariation introduced in \citet{kokoszka:taqqu:1995}. These measures of dependence are applicable to stable vectors with the index $\ag< 2$, and have been
studied in econometrics and statistical
physics contexts, see e.g., \citet{kokoszka:taqqu:1996}, \citet{levy:taqqu:2014}, and \citet{wolymanska:2015}. Exploring extreme value theory for functional data in this context might
be useful, but one must keep in mind that those measures are not symmetric.

\section*{Acknowledgement}
We Thank Professor Hong Miao of the Department of Finance and Real
Estate at Colorado State University for preprocessing the financial data used in Section \ref{ss:etf}. We thank Professor Joshua French of
the Department of Mathematical and Statistical Sciences at the University of Colorado Denver
for preprocessing the temperature data used in Section \ref{ss:hw}. We benefited from specific and detailed advice given by the reviewers. Kim and Kokoszka were partially supported by the United States National Science Foundation.

\section*{Supplementary material}
\label{SM}
The Supplementary material contains proofs of the theoretical results stated in the main
paper and additional empirical results.

\bibliographystyle{biometrika}
\bibliography{ecc}

\newpage

 \setcounter{page}{1}

\centerline{ \Large Supplementary material for ``Extremal correlation coefficient for functional data"}

\appendix

\appendixone
\section{Preliminary results}\label{s:pre}

In this section, we put together preliminary results needed to prove Theorem~\ref{t:cs}. Some of these results are known in the literature, and none of them are particularly profound or difficult to prove. However, these results allow us to streamline the exposition of proofs of the main result. Recall from~\eqref{eq:A} that $\mathcal A_r =\{(x,y)\in \mbB_0^2: \|(x,y)\|_{\mbB^2}  \ge r\}$, $r>0$, where $\|(x,y)\|_{\mbB^2} = \|x\|_{\mbB} \vee \|y\|_{\mbB}$.

\begin{lemma} \label{l:mu-conti}
Suppose $\mu$ is a measure in $M_0(\mbB^2)$ satisfying $\mu(t\cdot)=t^{-\ag}\mu(\cdot)$, $t>0$.  Then, $\mathcal A_r$ is a $\mu$--continuity set, i.e., $\mu(\partial \mathcal A_r)=0$.
\end{lemma}

\begin{proof}
We assume $\mu(\partial{\mathcal A_r})>0$ and get a contradiction. Since ${\mathcal A_r}\supset \bigcup_{n\ge 1} \partial(n^{1/\ag}{\mathcal A_r})$, it follows from the homogeneity property of $\mu$ that
\[
\mu({\mathcal A_r}) \ge \sum_{n=1}^{\infty} \mu(\partial(n^{1/\ag}{\mathcal A_r})) = \sum_{n=1}^{\infty} \mu(n^{1/\ag}\partial{\mathcal A_r}) =\sum_{n=1}^{\infty} n^{-1}\mu(\partial{\mathcal A_r})=\infty.
\]
It contradicts to the fact that $\mu$ is boundedly finite.
\end{proof}

Recall that $R= \|(X, Y)\|$, $R_i = \|(X_i, Y_i)\|$, and $R_{(k)}$ is the $k$th largest order statistic with the convention $R_{(1)} = \max \{ R_1, \ldots, R_n \}$. Let $b(n)$ be the quantile function such that ${\rm pr}(R>b(n))=n^{-1}$. Then, the following lemma holds by Proposition 3.1 of~\citet{segers:zhao:meinguet:2017}, and Theorem 4.1 and the proof of Theorem 4.2 of~\citet{resnick:2007}.

\begin{lemma} \label{l:R}

Let $M_{+}(0,\infty]$ be the space of Radon measures on $(0,\infty]$, and $\nu_{\ag}(r, \infty] = r^{-\ag}$. Also, let $\ep_x(A) = 1$ if $x \in A$ and $\ep_x(A) = 0$ if $x \notin A$. If $[X,Y]^{\top}$ is regularly varying in $L^2\times L^2$ according to Definition~\ref{d:RVinL2L2}, then

\noindent(i) $R$ is a nonnegative random variable whose distribution has a regularly varying tail with index $-\ag$,

\noindent(ii) $\frac{1}{k} \sum_{i=1}^n \ep_{R_i / b(n/k)} \convP \nu_{\ag}$, in $M_{+}(0,\infty]$,

\noindent(iii) $R_{(k)}/b(n/k) \convP 1$, in $[0,\infty)$,

\noindent(iv) $\frac{1}{k} \sum_{i=1}^n \ep_{R_i / R_{(k)}} \convP \nu_{\ag}$ in $M_{+}(0,\infty]$.

\end{lemma}

The following lemma is used to prove Lemmas~\ref{l:alM} and~\ref{l:cont-2}.

\begin{lemma} \label{l:cont}
Suppose $\ga_n$ converges vaguely to $\nu_\ag$ in $M_+(0,\infty]$. Then for any compact interval $K\subset (0,\infty]$,
\[
\int_K r^2\ga_n(dr) \to \int_K r^2 \nu_\ag(dr).
\]
\end{lemma}

\begin{proof}
Since the function $r\mapsto r^2 I_{K}$ is not continuous,
we use an approximation argument. Set $K= [a, b]$, for $0<a<b \le \infty$. Construct
compact intervals $K_j \searrow K$ and nonnegative continuous functions
$f_j$ such that $I_K \le f_j \le I_{K_j}$. By the triangle inequality,
\begin{align*}
\left | \int_K r^2 \ga_n(dr) - \int_K r^2 \nu_\ag(dr)\right |
&\le  \left | \int r^2 I_K(r) \ga_n(dr) - \int r^2 f_j(r) \ga_n (dr)\right |\\
& \ \ +
\left | \int r^2 f_j(r) \ga_n (dr)  - \int r^2 f_j(r) \nu_\ag  (dr)\right |\\
& \ \ +
\left | \int r^2 f_j(r) \nu_\ag  (dr)  - \int r^2 I_K(r) \nu_\ag(dr)\right |\\
&=: A_{n, j}^{(1)} + A_{n, j}^{(2)} + A_{j}^{(3)}.
\end{align*}
Fix $\tau > 0$. There is $j^\star$ such that for $j \ge j^\star$,
\[
A_{j}^{(3)} \le c\int \lb f_j(r) - I_K(r) \rb \nu_\ag(dr)
\le c \nu_\ag( K_j \setminus K^\circ) < \tau/2,
\]
where $c=b^2I_{b\neq \infty}+a^2I_{b=\infty}$.
Similarly $ A_{n, j}^{(1)} \le c \ga_n ( K_j \setminus K^\circ)$,
so for every fixed $j$,
\[
\limsup_{n\to \infty} A_{n, j}^{(1)} \le c
\limsup_{n\to \infty} \ga_n( K_j \setminus K^\circ)
\le c \nu_\ag( K_j \setminus K^\circ)
\]
because $K_j \setminus K^\circ$ is compact, cf. Proposition 3.12 in
\citet{resnick:1987}. Thus,
\[
\limsup_{n\to\infty}
\left | \int_K r^2 \ga_n(dr) - \int_K r^2 \nu_\ag(dr)\right |
\le  \tau + \limsup_{n\to\infty} A_{n, j^\star}^{(2)} = \tau.
\]
Since $\tau$ is arbitrary, we get the claim.
\end{proof}

The following lemma is used to prove Proposition~\ref{p:sglim}.

\begin{lemma} \label{l:cont_m0}
Under assumptions in Proposition~\ref{p:sglim}, we have that
\[
\int_{\cA_1} \lip x,y\rip n{\rm pr}({b(n)}^{-1}X \in dx, {b(n)}^{-1}Y \in dy) 
\to \int_{\cA_1} \lip x,y\rip \mu(dx, dy).
\]
\end{lemma}

\begin{proof}
Since the function $(x,y)\mapsto \lip x,y\rip I_{\cA_1}$ 
is not bounded, we employ an approximation argument, 
similar to that in Lemma~\ref{l:cont}. Let $\ga_n = n{\rm pr}({b(n)}^{-1}X \in \cdot, {b(n)}^{-1}Y \in \cdot)$, and consider a sequence $\{\cA_j\}_{j=1}^{\infty}$. By the triangle inequality, we have that
\begin{align*}
\left | \int_{\cA_1} \lip x,y\rip \ga_n(dx,dy) - \int_{\cA_1} \lip x,y\rip \mu(dx,dy)\right |
&\le  \left | \int_{\cA_1} \lip x,y\rip  \ga_n(dx,dy) - \int_{\cA_1 \setminus \cA_j} \lip x,y\rip  \ga_n (dx,dy)\right |\\
& \ \ +
\left | \int_{\cA_1 \setminus \cA_j} \lip x,y\rip  \ga_n (dx,dy)  - \int_{\cA_1 \setminus \cA_j} \lip x,y\rip  \mu  (dx,dy)\right |\\
& \ \ +
\left | \int_{\cA_1 \setminus \cA_j} \lip x,y\rip  \mu  (dx,dy)  - \int_{\cA_1} \lip x,y\rip  \mu(dx,dy)\right |\\
&=: B_{n, j}^{(1)} + B_{n, j}^{(2)} + B_{j}^{(3)}.
\end{align*}
Fix $\tau > 0$. It follows by~\eqref{eq:cov_pol} that there is $j^\star$ such that for $j \ge j^\star$,
\[
B_{j}^{(3)} \le \int_{\cA_j} \lip x,y\rip \mu(dx,dy) = \int_{r \ge j} r^2 \ag r^{-\ag-1} dr
\int_{\mbS_{L^2 \times L^2}} \lip \thg_X, \thg_Y \rip \Gg (d\thg_X,d\thg_Y) \le \frac{\ag}{\ag-2}j^{-\ag+2} < \tau/2.
\]
Similarly, it follows by change of variable that
\begin{align*}
B_{n, j}^{(1)} &\le \int_{\cA_j} \lip x,y\rip \ga_n(dx,dy) \\
&=  \int_{r \ge j}  \int_{\mbS_{L^2 \times L^2}}  r^2
\lip \thg_X, \thg_Y \rip n{\rm pr} \lp \frac{\|(X,Y)\|}{b(n)} \in dr, \frac{(X,Y)}{\|(X,Y)\|} \in (d\thg_X, d\thg_Y)\rp \\
& \le \int_{r \ge j} r^2 \ n{\rm pr} \lp \frac{\|(X,Y)\|}{b(n)} \in dr \rp.
\end{align*}
Then, by Lemmas~\ref{l:R} (i) and~\ref{l:cont}, we have that for each 
 $j \ge j^\star$
\[
\limsup_{n\to\infty} B_{n, j}^{(1)} \le  \int_{r \ge j} r^2 \nu_{\ag}(dr) = \frac{\ag}{\ag-2}j^{-\ag+2} < \tau/2.
\]
Thus, it follows from Lemma~\ref{l:mu-conti} and Lemma A.1 of~\citet{meinguet:segers:2010} that
\[
\limsup_{n\to\infty}
\left | \int_{\cA_1} \lip x,y\rip \ga_n(dx,dy) - \int_{\cA_1} \lip x,y\rip \mu(dx,dy)\right |
\le  \tau + \limsup_{n\to\infty} B_{n, j^\star}^{(2)} = \tau.
\]
Since $\tau$ is arbitrary, we get the claim.
\end{proof}

The following two lemmas are used to prove Lemma~\ref{l:22} and Proposition~\ref{p:bar}.

\begin{lemma} \label{l:alM}
Under Assumption~\ref{a:L2L2}, for any $M>0$,
\[
\frac{n}{k} E \lb \lp \frac{R}{b(n/k)} \rp^2 I_{R \ge Mb(n/k) }  \rb  \to \frac{\ag}{\ag-2}M^{2-\ag}.
\]
\end{lemma}

\begin{proof}

Observe that
\[
\frac{n}{k} E \lb \lp \frac{R}{b(n/k)} \rp^2 I_{R \ge Mb(n/k) }  \rb
=\int_M^{\infty} r^2  \frac{n}{k}{\rm pr}\lp \frac{R}{b(n/k)} \in dr\rp,
\]
and
\[
\frac{\ag}{\ag-2}M^{2-\ag} = \int_{M}^{\infty} r^2 \nu_{\ag}(dr).
\]
By Lemma~\ref{l:R} (i), we have that in $M_+(0,\infty]$
\[
\frac{n}{k}{\rm pr}\lp \frac{R}{b(n/k)} \in \cdot \rp \convv \nu_{\ag}.
\]
Therefore, we get the claim by Lemma~\ref{l:cont} with $K= [M, \infty]$.
\end{proof}

\begin{lemma} \label{l:cont-2}
The function $h$ on $M_+(0,\infty]$ defined by $h(\ga) = \int_1^M r^2 \ga(dr)$ is continuous at $\nu_\ag$.
\end{lemma}
\begin{proof}
Suppose $\ga_n$ converges vaguely to $\nu_\ag$. Then, by Lemma~\ref{l:cont} with $K= [1, M]$, it can be shown that
\[
\lim_{n\to\infty} \int_1^M r^2 \ga_n(dr) = \int_1^M r^2 \nu_\ag(dr).
\]
\end{proof}

The following lemma is the key argument to prove Proposition~\ref{p:hat}.

\begin{lemma} \label{l:22}
Under Assumption~\ref{a:L2L2}, the following statements hold:
\begin{align}
\frac{1}{k}\sum_{i=1}^n \lp \frac{R_i}{R_{(k)}} \rp^{2}   I_{R_i \ge R_{(k)} }  &\convP \frac{\ag}{\ag-2}; \label{eq:2k}\\
\frac{1}{k}\sum_{i=1}^n \lp \frac{R_i}{b(n/k)} \rp^{2}   I_{R_i \ge b(n/k) } &\convP \frac{\ag}{\ag-2}. \label{eq:2b}
\end{align}

\end{lemma}
\begin{proof}
The proofs for~\eqref{eq:2k} and~\eqref{eq:2b} are almost the same, so we only prove~\eqref{eq:2k} to save space. Let $\hat{\ga}_{n,k}=\frac{1}{k} \sum_{i=1}^n \ep_{R_i / R_{(k)}}$, and recall that $\hat{\ga}_{n,k} \convP \nu_{\ag}$ (see Lemma~\ref{l:R} (iv)). Since
\[
\frac{1}{k}\sum_{i=1}^n \lp \frac{R_i}{R_{(k)}} \rp^{2}   I_{R_i \ge R_{(k)} }= \int_1^\infty r^2 \hat{\ga}_{n,k}(dr),
\]
we need to show that
\[
\int_1^\infty r^2  \hat{\ga}_{n,k}(dr)
\convP \int_1^\infty r^{2} \nu_\ag(dr)=\frac{\ag}{\ag-2}.
\]
To prove this convergence, we use the second converging together theorem, Theorem 3.5
in \citet{resnick:2007}, (also stated as Theorem 3.2 of \citet{billingsley:1999}).

Let
\begin{align*}
V_{n,k} = \int_1^\infty r^2 \hat{\ga}_{n,k}(dr),  \ \ \
&  V= \int_1^\infty r^2 \nu_\ag(dr);\\
V_{n,k}^{(M)} = \int_1^M r^2 \hat{\ga}_{n,k}(dr),  \ \ \
& V^{(M)}= \int_1^M r^2 \nu_\ag(dr).
\end{align*}
To show the desired convergence $V_{n,k} \convP V$ (equivalently, $V_{n,k} \convd V$), we must verify that
\begin{equation} \label{eq:VnM}
\forall \ M > 1, \ \ \ V_{n,k}^{(M)} \convd V^{(M)}, \ \ \ \ {\rm as} \ n \to\infty;
\end{equation}
\begin{equation} \label{eq:VM}
V^{(M)} \convd V, \ \ \ {\rm as} \ M\to\infty;
\end{equation}
\begin{equation} \label{eq:bill}
\forall \ \eg > 0, \ \ \ \lim_{M\to\infty} \limsup_{n\to\infty}
 {\rm pr}\lp | V_{n,k}^{(M)}- V_{n,k} | > \eg \rp = 0.
\end{equation}
Convergence \eqref{eq:VnM} follows from Lemma~\ref{l:R} (iv)
and Lemma~\ref{l:cont-2}. Convergence \eqref{eq:VM} holds since for $\ag>2$
\[
\int_M^\infty r^2 \nu_\ag(dr)=\int_M^\infty r^2 \ag r^{-\ag-1} dr =\frac{\ag}{\ag-2} M^{2-\ag} \to 0, \ \ \ \ {\rm as}\ M \to \infty.
\]
It remains to show that $\forall \eg>0$,
\[
\lim_{M\to\infty} \limsup_{n\to\infty}
 {\rm pr}\lp | V_{n,k}^{(M)}- V_{n,k} | > \eg \rp =  \lim_{M\to\infty} \limsup_{n\to\infty}
 {\rm pr}\lp \int_M^\infty r^2 \hat\ga_{n,k}(dr) > \eg \rp = 0.
\]

Fix $\eg> 0$ and $\eta> 0$. Observe that
\[
{\rm pr}\lp \int_M^\infty r^2 \hat\ga_{n,k}(dr) > \eg \rp
 \le Q_1(n) + Q_2(n),
\]
where
\[
Q_1(n) = {\rm pr}\lp \int_M^\infty r^2 \hat\ga_{n,k}(dr) > \eg, \
\left | \frac{R_{(k)}}{b(n/k)} - 1 \right |  < \eta \rp,\ \
Q_2(n) = {\rm pr}\lp \left | \frac{R_{(k)}}{b(n/k)} - 1 \right |  \ge  \eta\rp.
\]
By Lemma~\ref{l:R} (iii), $ \limsup_{n\to\infty} Q_2(n) = 0$. For $Q_1(n)$, we start with the bound
\begin{align*}
Q_1(n)
&\le {\rm pr}\lp  \int_M^\infty r^2 \hat\ga_{n,k}(dr) > \eg, \
\frac{R_{(k)}}{b(n/k)} > 1- \eta \rp\\
&= {\rm pr} \lp  \int_M^\infty r^2 \frac{1}{k} \sum_{i=1}^n
\ep_{R_i/ R_{(k)}}(dr)  > \eg, \
\frac{R_{(k)}}{b(n/k)} > 1- \eta \rp.
\end{align*}
Conditions $R_i/ R_{(k)} > M$ and ${R_{(k)}}/{b(n/k)} > 1- \eta$
imply ${R_{i}}/{b(n/k)} > M(1- \eta)$, so
\begin{align*}
Q_1(n)
&\le {\rm pr} \lp  \int_{M(1-\eta)}^\infty r^2 \frac{1}{k} \sum_{i=1}^n
\ep_{R_i/b(n/k)}(dr)   > \eg \rp\\
&= {\rm pr} \lp  \frac{1}{k} \sum_{i=1}^n  \lp \frac{R_i}{b(n/k)} \rp^2
I_{R_i \ge M(1-\eta)b(n/k) } >\eg\rp.
\end{align*}
Then, it follows from Markov's inequality and Lemma~\ref{l:alM} that
\[
Q_1(n) \le \frac{1}{\eg} \frac{n}{k} E \lb \lp \frac{R_1}{b(n/k)} \rp^2
I_{R_1 \ge M(1-\eta)b(n/k) } \rb  \to \frac{1}{\eg} \frac{\ag}{\ag-2} \{M(1-\eta)\}^{2-\ag}, \ \ \ {\rm as} \ n \to \infty.
\]
This bound goes to 0 as $M \to \infty$ since $\ag>2$.
\end{proof}

The following lemma follows from Theorem 3.8 of~\citet{mcdiarmid:1998}. It states a Bernstein type inequality, which is the key technique to prove Proposition~\ref{p:bar}.

\begin{lemma} \label{l:Bern}
Let $\bZ_n=(Z_1, \ldots, Z_n)$ with the $Z_i$ taking values in a
Lebesgue measurable subset $\cZ$ of an Euclidean space. Let $f$ be a real-valued function defined on $\cZ^n$. For $(z_1, \ldots, z_i) \in \cZ^{i}$, $1 \le i \le n$, put
\begin{equation} \label{eq:gi}
g_i(z_1, \ldots, z_i):=E\lb f(\bZ_n)| Z_j=z_j, 1 \le j \le i\rb -E\lb f(\bZ_n)| Z_j=z_j, 1 \le j \le i-1\rb.
\end{equation}
Define the maximum deviation by
\begin{equation} \label{eq:b}
b := \max_{1\le i\le n}\sup_{(z_1, \ldots, z_i) \in \cZ^i} g_i(z_1, \ldots, z_i),
\end{equation}
and define the supremum sum of variances by
\begin{equation} \label{eq:v}
\hat{v} := \sup_{(z_1, \ldots, z_n) \in \cZ^n} \sum_{i=1}^n \var\lb g_i(z_1, \ldots, z_{i-1}, Z_i^{\prime})\rb,
\end{equation}
where $Z_i^{\prime}$ is an independent copy of $Z_i$ conditional on $Z_j=z_j$, $1 \le j \le i-1$. If $b$ and $\hat{v}$ are finite, then for any $\eg \ge 0$,
\[
{\rm pr} \lp f(\bZ_n) - E[f(\bZ_n)] \ge t \rp \le \exp \lp \frac{-\eg^2}{2(\hat{v}+b\eg/3)}  \rp.
\]
\end{lemma}

\clearpage

\appendixtwo
\section{Proof of Theorem~\ref{t:cs} in Section~\ref{s:ecc}}\label{s:pr}

Recall~\eqref{eq:ec}, i.e., the definition:
\[
\hat{\sg}_{n,k} =\frac{1}{k}\sum_{i=1}^n \lip \frac{X_i}{R_{(k)}}, \frac{Y_i}{R_{(k)}} \rip I_{R_i \ge R_{(k)} }.
\]
To prove the consistency of $\hat{\sg}_{n,k}$ for the extremal covariance $\sg_{XY}$, we consider the following sequence of random variables
\begin{equation} \label{eq:sg}
\sg_{n,k} :=\frac{1}{k}\sum_{i=1}^n \lip \frac{X_i}{b(n/k)}, \frac{Y_i}{b(n/k)} \rip I_{R_i \ge b(n/k) }.
\end{equation}
We emphasize that $\sg_{n,k}$ is not observable since $b(\cdot)$ is unknown. However, $b(n/k)$ can be estimated by its consistent estimator $R_{(k)}$, and it can be shown that replacing $b(n/k)$ by $R_{(k)}$ ensures that the difference between $\sg_{n,k}$ and $\hat{\sg}_{n,k}$ is asymptotically negligible, which will be shown in Proposition~\ref{p:hat}. Thus, the key argument for establishing the consistency is to show that $\sg_{n,k}$ converges in probability to $\sg_{XY}$, which is proven in the following proposition.

\begin{proposition} \label{p:bar}
Under Assumption~\ref{a:L2L2},
\[
\sg_{n,k} \convP \sg_{XY}.
\]
\end{proposition}

\begin{proof}

Set
\begin{equation} \label{eq:sgbar}
\bar{\sg}_{n,k} := E\lb \lip \frac{X_1}{b(n/k)}, \frac{Y_1}{b(n/k)} \rip \Bigg|\|X_1\| \vee \|Y_1\| >b(n/k)\rb.
\end{equation}
Then, by Proposition~\ref{p:sglim}, $\bar{\sg}_{n,k} \to \sg_{XY}$, so it remains to show that $|\sg_{n,k} - \bar{\sg}_{n,k}| \convP 0$.

Let $\bZ_n = (Z_1, \ldots, Z_n)$, where $Z_i = (X_i, Y_i)$, and $\bz_n = (z_1, \ldots, z_n)$, where $z_i = (x_i, y_i)$, for $1 \le i \le n$. Consider a map $f: (L^2 \times L^2)^n \to \mbR$ defined by
\[
f(\bz_n) := \lmo \frac{1}{k}\sum_{i=1}^n \lip \frac{x_i}{b(n/k)}, \frac{y_i}{b(n/k)} \rip I_{r_i \ge b(n/k) }- \frac{n}{k}E\lb \lip \frac{X_1}{b(n/k)}, \frac{Y_1}{b(n/k)} \rip I_{R_1 >b(n/k)}\rb \rmo.
\]
Then, we have that
\[
|\sg_{n,k} - \bar{\sg}_{n,k}| = f(\bZ_n) - E[f(\bZ_n)] + E[f(\bZ_n)].
\]
We aim to show that $f(\bZ_n) - E[f(\bZ_n)] \convP 0$ and $E[f(\bZ_n)] \to 0$.

To establish the convergence, $f(\bZ_n) - E[f(\bZ_n)] \convP 0$, we use the Bernstein type concentration inequality in Lemma~\ref{l:Bern}. Since the $(X_i,Y_i)$ are independent, the deviation function in~\eqref{eq:gi} has the following form
\[
g_i(z_1, \ldots, z_i)= E \lb f(z_1, \ldots, z_{i-1},z_i, Z_{i+1}, \ldots, Z_n)-f(z_1, \ldots, z_{i-1},Z_i, Z_{i+1}, \ldots, Z_n) \rb.
\]
Then, using the fact that $\lmo |x|- |y| \rmo \le |x-y|$, we have that
\begin{align*}
g_i(z_1, \ldots, z_i)
&\le \frac{1}{k} E\lb \lmo  \lip \frac{x_i}{b(n/k)}, \frac{y_i}{b(n/k)} \rip I_{r_i \ge b(n/k) } - \lip \frac{X_i}{b(n/k)}, \frac{Y_i}{b(n/k)} \rip I_{R_i \ge b(n/k) }       \rmo \rb \\
&\le \frac{1}{k} \lbr \frac{|\lip x_i, y_i\rip|}{b(n/k)^2} + \frac{k}{n} \frac{n}{k} E \lb \lp \frac{R_i}{b(n/k)} \rp^2 I_{R_i \ge b(n/k) }  \rb \rbr\\
&\le \frac{1}{k} \lbr \frac{\|x_i\| \| y_i\|}{b(n/k)^2} + \frac{n}{k} E \lb \lp \frac{R_i}{b(n/k)} \rp^2 I_{R_i \ge b(n/k) }  \rb \rbr.
\end{align*}
Since $(x_i,y_i) \in L^2\times L^2$ and $\frac{n}{k} E \lb \lp {R_i}/{b(n/k)} \rp^2 I_{R_i \ge b(n/k) }  \rb \to \ag/(\ag-2)$ by Lemma~\ref{l:alM},  we have that $g_{i}(z_1, \ldots, z_i) \le {c_1}/{k}$, for some constant $c_1>0$. Therefore, the maximum deviation $b$ in~\eqref{eq:b} is bounded by $c_1/k$.

Next we investigate the upper bound for the sum of variances $\hat{v}$ in~\eqref{eq:v}. Since $E[g_i(z_1, \ldots, z_{i-1}, Z_i^{\prime})]=0$ by the law of total probability, we have that
\begin{align*}
&\var\lb g_i(z_1, \ldots, z_{i-1}, Z_i^{\prime})\rb \\
&= E [g_i^2(z_1, \ldots, z_{i-1}, Z_i^{\prime})]\\
&= E \lb \lbr  f(z_1, \ldots, z_{i-1},Z_i^{\prime}, Z_{i+1}, \ldots, Z_n)-f(z_1, \ldots, z_{i-1},Z_i, Z_{i+1}, \ldots, Z_n) \rbr^2 \rb \\
&\le \frac{1}{k^2} E\lb \lbr  \lip \frac{X_i^{\prime}}{b(n/k)}, \frac{Y_i^{\prime}}{b(n/k)} \rip I_{R_i^{\prime} \ge b(n/k) } - \lip \frac{X_i}{b(n/k)}, \frac{Y_i}{b(n/k)} \rip I_{R_i \ge b(n/k) }       \rbr^2 \rb \\
&\le \frac{2}{k^2} E \lb \lip \frac{X_i}{b(n/k)}, \frac{Y_i}{b(n/k)} \rip^2 I_{R_i \ge b(n/k) }   \rb \\
&\le \frac{2}{k^2} \lbr  \frac{k}{n} \frac{n}{k} E \lb  \lp \frac{R_i}{b(n/k)} \rp^2  I_{R_i \ge b(n/k) }  \rb \rbr.
\end{align*}
It then again follows from Lemma~\ref{l:alM} that $\var\lb g_i(z_1, \ldots, z_{i-1}, Z_i^{\prime})\rb  \le c_2/(nk)$ for some $c_2>0$. Then the supremum sum of variances $\hat{v}$ is bounded above by $c_2/k$. Therefore by Lemma~\ref{l:Bern}, for any $\eg>0$
\[
{\rm pr} \lp f(\bZ_n) - E[f(\bZ_n)]  \ge \eg \rp \le \exp \lp \frac{-k\eg^2}{c_1+c_2\eg/3} \rp.
\]
If we apply this inequality to $-f(\bZ_n)$, then we obtain the following `two-sided' inequality
\[
{\rm pr} \lp |f(\bZ_n) - E[f(\bZ_n)]|  \ge \eg \rp \le 2\exp \lp \frac{-k\eg^2}{c_1+c_2\eg/3} \rp.
\]
From this, we obtain that  $f(\bZ_n) - E[f(\bZ_n)] \convP 0$.

Next, to show $E[f(\bZ_n)] \to 0$, we set, for $1 \le i \le n$
\[
\Delta_i =  \lip \frac{X_i}{b(n/k)}, \frac{Y_i}{b(n/k)} \rip I_{R_i \ge b(n/k) }- E\lb \lip \frac{X_1}{b(n/k)}, \frac{Y_1}{b(n/k)} \rip I_{R_1 >b(n/k)}\rb.
\]
Then, we have that
\begin{align*}
E \lb f(\bZ_n) \rb = \frac{n}{k} E \lb \lmo \frac{1}{n} \sum_{i=1}^n \Delta_i \rmo \rb &\le \frac{n}{k} \lbr E \lb \lp \frac{1}{n} \sum_{i=1}^n \Delta_i \rp^2 \rb \rbr^{1/2}\\
&= \frac{n}{k} \lbr E \lb   \frac{1}{n^2} \sum_{i=1}^n \Delta_i^2 + \frac{1}{n^2} \sum_{i \neq j} \Delta_i\Delta_j \rb \rbr^{1/2}.
\end{align*}
Since the $\Delta_i$ are independent, $E[\Delta_i\Delta_j]=0$, for $i \neq j$. Therefore,
\begin{align*}
&E \lb f(\bZ_n) \rb \\
&\le \frac{n^{1/2}}{k} \lbr E \lb    \Delta_1^2  \rb \rbr^{1/2} \\
&= \frac{n^{1/2}}{k} \lbr  E  \lb \lp \lip \frac{X_1}{b(n/k)}, \frac{Y_1}{b(n/k)} \rip I_{R_1 \ge b(n/k) } - E\lb \lip \frac{X_1}{b(n/k)}, \frac{Y_1}{b(n/k)} \rip I_{R_1 >b(n/k)}\rb  \rp^2\rb \rbr^{1/2}\\
& = \frac{n^{1/2}}{k} \lbr  \var \lb  \lip \frac{X_1}{b(n/k)}, \frac{Y_1}{b(n/k)} \rip I_{R_1 \ge b(n/k) } \rb \rbr^{1/2} \\
& \le \frac{n^{1/2}}{k} \lbr  E  \lb \lip \frac{X_1}{b(n/k)}, \frac{Y_1}{b(n/k)} \rip^2 I_{R_1 \ge b(n/k) } \rb \rbr^{1/2}\\
& \le  \frac{n^{1/2}}{k}  \lbr E \lb  \lp \frac{R_1}{b(n/k)} \rp^2  I_{R_1 \ge b(n/k) }  \rb \rbr^{1/2}.
\end{align*}
Therefore, by Lemma~\ref{l:alM} we have that
\[
E \lb f(\bZ_n) \rb \le \frac{n^{1/2}}{k}  \lbr  \frac{k}{n} \frac{n}{k} E \lb  \lp \frac{R_1}{b(n/k)} \rp^2  I_{R_1 \ge b(n/k) }  \rb \rbr^{1/2} \le \frac{c_3}{k^{1/2}},
\]
for some $c_3>0$, which completes the proof.
\end{proof}

\begin{proposition} \label{p:hat}
Under Assumption~\ref{a:L2L2},
\[
|\hat{\sg}_{n,k} - \sg_{n,k} | \convP 0.
\]
\end{proposition}

\begin{proof}
Consider the following decomposition
\[
|\hat{\sg}_{n,k} - \sg_{n,k} | \le P_1(n) + P_2(n),
\]
where
\begin{align*}
&P_1(n) :=\lmo \frac{1}{k}\sum_{i=1}^n \lip \frac{X_i}{R_{(k)}}, \frac{Y_i}{R_{(k)}} \rip  \lbr I_{R_i \ge R_{(k)} } - I_{R_i \ge b(n/k) }\rbr \rmo,\\
&P_2(n) :=\lmo \frac{1}{k}\sum_{i=1}^n \lbr \lip \frac{X_i}{R_{(k)}}, \frac{Y_i}{R_{(k)}} \rip - \lip \frac{X_i}{b(n/k)}, \frac{Y_i}{b(n/k)} \rip \rbr
 I_{R_i \ge b(n/k) } \rmo.
\end{align*}
We will show that each of the two parts goes to 0. We first focus on $P_1(n)$. Observe that
\begin{align*}
P_1(n) &\le \lp\frac{b(n/k)}{R_{(k)}} \rp^2  \frac{1}{k}\sum_{i=1}^n  \lmo \lip \frac{X_i}{R_i}, \frac{Y_i}{R_i} \rip \rmo \lp \frac{R_i}{b(n/k)} \rp^{2} \lmo  I_{R_i \ge R_{(k)} } - I_{R_i \ge b(n/k) } \rmo \\
&\le \lp\frac{b(n/k)}{R_{(k)}} \rp^2  \frac{1}{k}\sum_{i=1}^n  \lp \frac{R_i}{b(n/k)} \rp^{2} \lmo  I_{R_i \ge R_{(k)} } - I_{R_i \ge b(n/k) } \rmo \\
& =  \lp\frac{b(n/k)}{R_{(k)}} \rp^2 \lmo \frac{1}{k}\sum_{i=1}^n \lp \frac{R_i}{b(n/k)} \rp^{2}   I_{R_i \ge R_{(k)} } - \frac{1}{k}\sum_{i=1}^n \lp \frac{R_i}{b(n/k)} \rp^{2}   I_{R_i \ge b(n/k) } \rmo \\
& =  \lmo \frac{1}{k}\sum_{i=1}^n \lp \frac{R_i}{R_{(k)}} \rp^{2}   I_{R_i \ge R_{(k)} } - \lp\frac{b(n/k)}{R_{(k)}} \rp^2  \frac{1}{k}\sum_{i=1}^n \lp \frac{R_i}{b(n/k)} \rp^{2}   I_{R_i \ge b(n/k) } \rmo
\end{align*}
Then, by Lemma~\ref{l:R} (iii), we have that $\lp {b(n/k)}/{R_{(k)}} \rp^2 \convP 1$. By Lemma~\ref{l:22} that $\frac{1}{k}\sum_{i=1}^n \lp{R_i}/{R_{(k)}} \rp^{2}   I_{R_i \ge R_{(k)} } \convP \ag/(\ag-2)$ and $\frac{1}{k}\sum_{i=1}^n \lp {R_i}/{b(n/k)} \rp^{2}   I_{R_i \ge b(n/k) }\convP \ag/(\ag-2)$. Therefore, we have that  $P_1(n) \convP 0$.

Now we work on $P_2(n)$. Observe that
\begin{align*}
P_2(n) & = \lmo \frac{1}{k}\sum_{i=1}^n \lip \frac{X_i}{R_i}, \frac{Y_i}{R_i} \rip R_i^2 \lp \frac{1}{R_{(k)}^2} - \frac{1}{b(n/k)^2}\rp
 I_{R_i \ge b(n/k) } \rmo \\
& \le  \lmo \frac{b(n/k)^2}{R_{(k)}^2} - 1 \rmo  \frac{1}{k}\sum_{i=1}^n  \lmo \lip \frac{X_i}{R_i}, \frac{Y_i}{R_i}  \rip \rmo \lp \frac{R_i}{b(n/k)} \rp^{2}
 I_{R_i \ge b(n/k)} \\
 & \le  \lmo \frac{b(n/k)^2}{R_{(k)}^2} - 1 \rmo  \frac{1}{k}\sum_{i=1}^n \lp \frac{R_i}{b(n/k)} \rp^{2} I_{R_i \ge b(n/k)}.
\end{align*}
By Lemma~\ref{l:alM}, we have that $\frac{1}{k}\sum_{i=1}^n \lp {R_i}/{b(n/k)} \rp^{2} I_{R_i \ge b(n/k)} = O_P(1)$, and by Lemma~\ref{l:R} (iii), we have that $b(n/k) / R_{(k)} \convP 1$. Thus, $P_2(n) \convP 0$.
\end{proof}

\noindent{Proof of Theorem~\ref{t:cs}.} It follows from Propositions~\ref{p:bar} and~\ref{p:hat}.

\clearpage

\appendixthree
\section{Proof of Lemma~\ref{l:sim1} in Section~\ref{s:pw} and its extension to randomly sampled weights} \label{s:pr_sim}

\allowdisplaybreaks
\noindent{\sc Proof of Lemma~\ref{l:sim1}:} \
We begin by noting that there exists an increasing sequence $b(n) \to \infty$ and $\nu$ in $M_+(\mbR_+^2)$ such that
\begin{equation} \label{eq:dRV2}
n{\rm pr} \lp \frac{(|Z_1|, |Z_2|)}{b(n)} \in \cdot \rp \convv \nu.
\end{equation}
Here, $\nu$ is defined up to a multiplicative constant. Since $Z_1$ and $Z_2$ are independent, $\nu$ has its mass only on the axes.

The specific form of $\nu$ can be given by, for $\bx = [z_1, z_2]^{\top}$,
\begin{equation} \label{eq:cnu}
\nu([0, \bx]^c) = c(z_1)^{-\ag} +c(z_2)^{-\ag},
\end{equation}
where $c = 1/(1+ (1-\rho^2)^{\ag/2})$. To see this, observe that
\begin{align*}
n{\rm pr}(\| Z_1 \phi_1 \|  \vee  \| \rho Z_1 \phi_1+(1-\rho^2)^{1/2} Z_2 \phi_2 \| > b(n)) = n{\rm pr}(|Z_1|  \vee  (\rho^2 Z_1^2 + (1-\rho^2) Z_2^2)^{1/2} > b(n)).
\end{align*}
By~\eqref{eq:mu-s}, it follows that $\nu(\{(z_1,z_2): |z_1| \vee ( \rho^2 z_1^2 + (1-\rho^2) z_2^2 )^{1/2} >1 \})$ must be 1. Note that $\{(z_1,z_2): |z_1| \vee ( \rho^2 z_1^2 + (1-\rho^2) z_2^2 )^{1/2} >1 \} = \{(z_1,z_2): z_1>1 \ {\rm or}\ z_2>1/(1-\rho^2)^{1/2} \}$. Using this, we have that
\begin{align*}
\nu(\{(z_1,z_2): |z_1| \vee ( \rho^2 z_1^2 + (1-\rho^2) z_2^2 )^{1/2} >1 \})
=&c + c (1-\rho^2)^{\ag/2} = 1.
\end{align*}

We claim that
\begin{align} \label{eq:sxy}
    & \sg_{XY}=\rho  \frac{c\ag}{\ag-2};\\ \label{eq:sx}
    & \sg_{X}^2 =\frac{c\ag}{\ag-2};\\ \label{eq:sy}
    & \sg_{Y}^2 =\lbr \rho^2 + (1-\rho^2)^{\ag/2} \rbr \frac{c\ag}{\ag-2}.
\end{align}

We first work on~\eqref{eq:sxy}. Since the terms with the $N_j$ do not affect the extremal behavior of $X$ and $Y$, we have that by Proposition~\ref{p:sglim}
\begin{align*}
&\sg_{XY} \\
&= \lim_{n \to \infty} E \lb \lip \frac{Z_1 \phi_1}{b(n)}, \frac{\rho Z_1 \phi_1+(1-\rho^2)^{1/2} Z_2 \phi_2}{b(n)}\rip  \Bigg| \| Z_1 \phi_1 \|  \vee  \| \rho Z_1 \phi_1+(1-\rho^2)^{1/2} Z_2 \phi_2 \| > b(n)\rb  \\
&= \lim_{n \to \infty} \frac{1}{{\rm pr}(\| Z_1 \phi_1 \|  \vee  \| \rho Z_1 \phi_1+(1-\rho^2)^{1/2} Z_2 \phi_2 \| > b(n))} \times \\
&\ \ \ \ \ \ \ \ \ \ \ \ \ \ \ \ \ \ E \lb \lip \frac{Z_1 \phi_1}{b(n)}, \frac{\rho Z_1 \phi_1+(1-\rho^2)^{1/2} Z_2 \phi_2}{b(n)}\rip  I_{\| Z_1 \phi_1 \|  \vee  \| \rho Z_1 \phi_1+(1-\rho^2)^{1/2} Z_2 \phi_2 \| > b(n)}\rb  \\
&= \lim_{n \to \infty} \frac{1}{{\rm pr}(\| Z_1 \phi_1 \|  \vee  \| \rho Z_1 \phi_1+(1-\rho^2)^{1/2} Z_2 \phi_2 \| > b(n))}   E \lb  \rho \frac{Z_1^2 }{b(n)^2} I_{|Z_1| \vee (\rho^2 Z_1^2 +(1-\rho^2)Z_2^2 )^{1/2} > b(n)} \rb.
\end{align*}

It then follows from~\eqref{eq:dRV2} and~\eqref{eq:mu-s} that
\begin{align*}
\sg_{XY} &= \lim_{n \to \infty} n E \lb  \rho \frac{Z_1^2 }{b(n)^2} I_{|Z_1| \vee (\rho^2 Z_1^2 +(1-\rho^2)Z_2^2 )^{1/2} > b(n)} \rb \\
&= \lim_{n \to \infty} \int_{\mbR^2_+} \rho z_1^2 I_{ |z_1| \vee ( \rho^2 z_1^2 + (1-\rho^2) z_2^2 )^{1/2} >1} n{\rm pr}\lp \frac{|Z_1|}{b(n)} \in dz_1, \frac{|Z_2|}{b(n)}\in dz_2 \rp \\
&= \int_{\mbR_+^2}  \rho z_1^2 I_{ |z_1| \vee ( \rho^2 z_1^2 + (1-\rho^2) z_2^2 )^{1/2} >1} \ \nu(dz_1, dz_2) \\
& = \int_{\mbR_+}  \rho z_1^2 I_{ \{(z_1, 0): z_1>1\} } \ c\nu_{\ag}(dz_1) + \int_{\mbR_+}  \rho z_1^2 I_{ \{(0, z_2): z_2>1/(1-\rho^2)^{1/2} \}} \ c\nu_{\ag}(dz_2) \\
&= \int_{1}^{\infty}  \rho z_1^2 \ c\nu_{\alpha}(dz_1) + 0 = \rho \frac{c\ag}{\ag-2}.
\end{align*}

Analogously, for~\eqref{eq:sx} we can show that
\begin{align*}
&\sg_{X}^2 \\
&= \lim_{n \to \infty} E \lb \lip \frac{Z_1 \phi_1}{b(n)}, \frac{Z_1 \phi_1}{b(n)}\rip  \Bigg| \| Z_1 \phi_1 \|  \vee  \| \rho Z_1 \phi_1+(1-\rho^2)^{1/2} Z_2 \phi_2 \| > b(n)\rb  \\
&= \lim_{n \to \infty} \frac{1}{{\rm pr}(\| Z_1 \phi_1 \|  \vee  \| \rho Z_1 \phi_1+(1-\rho^2)^{1/2} Z_2 \phi_2 \| > b(n))}  E \lb  \frac{Z_1^2 }{b(n)^2} I_{|Z_1| \vee (\rho^2 Z_1^2 +(1-\rho^2)Z_2^2 )^{1/2} > b(n)} \rb \\
&= \lim_{n \to \infty} nE \lb  \frac{Z_1^2 }{b(n)^2} I_{|Z_1| \vee (\rho^2 Z_1^2 +(1-\rho^2)Z_2^2 )^{1/2} > b(n)} \rb \\
&=\frac{c\ag}{\ag-2}.
\end{align*}

Next, we work on~\eqref{eq:sy}. Observe that
\begin{align*}
&\sg_{Y}^2 \\
&= \lim_{n \to \infty} E \lb   \frac{\|\rho Z_1 \phi_1+(1-\rho^2)^{1/2} Z_2 \phi_2 \|^2}{b(n)^2}  \Bigg| \| Z_1 \phi_1 \|  \vee  \| \rho Z_1 \phi_1+(1-\rho^2)^{1/2} Z_2 \phi_2 \| > b(n)\rb  \\
&= \lim_{n \to \infty} \frac{1}{{\rm pr}(\| Z_1 \phi_1 \|  \vee  \| \rho Z_1 \phi_1+(1-\rho^2)^{1/2} Z_2 \phi_2 \| > b(n))} \times \\
&\ \ \ \ \ \ \ \ \ \ \ \ \ \ \ \ \ \ \ \ \ \ \ \ \ \ \ E \lb  \frac{\rho^2 Z_1^2 +(1-\rho^2) Z_2^2 }{b(n)^2} I_{|Z_1| \vee (\rho^2 Z_1^2 +(1-\rho^2)Z_2^2 )^{1/2} > b(n)} \rb.
\end{align*}
Then, again it follows from~\eqref{eq:dRV2} and~\eqref{eq:mu-s} that
\begin{align*}
\sg_{Y}^2 &= \lim_{n \to \infty} nE \lb  \frac{\rho^2 Z_1^2 +(1-\rho^2) Z_2^2 }{b(n)^2} I_{|Z_1| \vee (\rho^2 Z_1^2 +(1-\rho^2)Z_2^2 )^{1/2} > b(n)} \rb   \\
&= \lim_{n \to \infty} \int_{\mbR^2_+} \lbr \rho^2 z_1^2 + (1-\rho^2)z_2^2 \rbr I_{ |z_1| \vee ( \rho^2 z_1^2 + (1-\rho^2) z_2^2 )^{1/2} >1} n{\rm pr}\lp \frac{|Z_1|}{b(n)} \in dz_1, \frac{|Z_2|}{b(n)}\in dz_2 \rp \\
&= \int_{\mbR_+^2}  \lbr \rho^2 z_1^2 + (1-\rho^2)z_2^2 \rbr I_{ |z_1| \vee ( \rho^2 z_1^2 + (1-\rho^2) z_2^2 )^{1/2} >1} \ \nu(dz_1, dz_2) \\
& = \int_{\mbR_+}  \rho^2 z_1^2 I_{ \{(z_1, 0): z_1>1\} } \ c\nu_{\ag}(dz_1) + \int_{\mbR_+}  (1-\rho^2)z_2^2 I_{ \{(0, z_2): z_2>1/(1-\rho^2)^{1/2} \}} \ c\nu_{\ag}(dz_2) \\
&= \int_{1}^{\infty}  \rho^2 z_1^2 \ c\nu_{\alpha}(dz_1) + \int_{1/(1-\rho^2)^{1/2}}^{\infty}   (1-\rho^2)z_2^2 \ c\nu_{\alpha}(dz_2) \\
&= \rho^2\frac{c\ag}{\ag-2} + (1-\rho^2)^{\ag/2}\frac{c\ag}{\ag-2} = \{\rho^2+(1-\rho^2)^{\ag/2}\}\frac{c\ag}{\ag-2}.
\end{align*}
This completes the proof of Lemma~\ref{l:sim1}.

An alternative approach to proving Lemma~\ref{l:sim1} is to identify the joint angular measure $\Gg$ of $[X,Y]^{\top}$ and then apply it in the decomposition given by~\eqref{eq:cov_pol}. Since only one of $Z_1$, $Z_2$ can be extreme, $\Gg$ concentrates on $(\phi_1, \rho\phi_1)$ with mass $\nu(\{(z_1, z_2) : |z_1| \vee |\rho z_1| >1\})=c$ or $(\bzero, \phi_2)$ with mass $\nu(\{(z_1, z_2) : 0 \vee |(1-\rho^2)^{1/2} z_2| >1\})=c (1-\rho^2)^{\ag/2}$. Therefore, we have that
\[
\Gg = c \ep_{(\phi_1, \rho\phi_1)} + c (1-\rho^2)^{\ag/2} \ep_{(\bzero, \phi_2)}.
\]
Computing $\Gg$ in~\eqref{eq:cov_pol}, we obtain~\eqref{eq:sxy}. Similarly,~\eqref{eq:sx} and~\eqref{eq:sy} follow.

\medskip\medskip

We now extend Lemma~\ref{l:sim1} by considering randomly occurring
weights, so each direction in the function space can contribute either
a heavy-tailed or a light-weight component.

\begin{lemma} \label{l:sim2}
Let $\{A_i\}$ and $\{B_i\}$ be independent sequences of   iid Bernoulli
random variables with $P(A_i=1)=p_A$ and $P(B_i=1)=p_B$, independent of the $Z_i$.  Put
\[
X(t) = \sum_{i=1}^2 \phi_i(t)\{Z_iA_i + N_i(1-A_i)\}, \quad
Y(t) = \sum_{i=1}^2 \phi_i(t)\{Z_iB_i + N_i(1-B_i)\}.
\]
Then, under assumptions in Lemma~\ref{l:sim1}, we have that
\[
\rho_{XY} = (p_A)^{1/2}(p_B)^{1/2}.
\]
\end{lemma}

\begin{proof}

We will show that
\[
\sg_{X}^2 =\frac{c\ag}{\ag-2}p_A; \quad
\sg_{Y}^2 =\frac{c\ag}{\ag-2}p_B; \quad
\sg_{XY}=\frac{c\ag}{\ag-2}p_Ap_B.
\]
We first work on $\sg_{X}^2$. With the choice of $b(n)$ defined by $n^{-1} = {\rm pr}((A_1^2Z_1^2 + A_2^2Z_2^2)^{1/2} \vee (B_1^2Z_1^2 + B_2^2Z_2^2)^{1/2}  > b(n))$, it follows from the law of total expectation that
\begin{align*}
\sg_{X}^2
&= \lim_{n \to \infty} n E \lb  \frac{A_1^2Z_1^2 + A_2^2Z_2^2}{b(n)^2} I_{(A_1^2Z_1^2 + A_2^2Z_2^2)^{1/2} \vee (B_1^2Z_1^2 + B_2^2Z_2^2)^{1/2} > b(n)} \rb \\
&= \lim_{n \to \infty} n E \lb  \frac{Z_1^2 + Z_2^2}{b(n)^2} I_{(Z_1^2 + Z_2^2)^{1/2} \vee (B_1^2Z_1^2 + B_2^2Z_2^2)^{1/2} > b(n)} \rb \times p_A^2 \\
&+\lim_{n \to \infty} n E \lb  \frac{Z_1^2}{b(n)^2} I_{|Z_1| \vee (B_1^2Z_1^2 + B_2^2Z_2^2)^{1/2} > b(n)} \rb \times p_A(1-p_A) \\
& + \lim_{n \to \infty} n E \lb  \frac{Z_2^2}{b(n)^2} I_{|Z_2| \vee (B_1^2Z_1^2 + B_2^2Z_2^2)^{1/2} > b(n)} \rb \times (1-p_A)p_A.
\end{align*}
Then, using~\eqref{eq:dRV2} and vague convergence, it simplifies to
\begin{align*}
\sg_{X}^2
& = \lb \int_{\mbR_+}  z_1^2  I_{ \{(z_1, 0): z_1>1\} } \ c\nu_{\ag}(dz_1) + \int_{\mbR_+}  z_2^2 I_{ \{(0, z_2): z_2>1 \}} \ c\nu_{\ag}(dz_2) \rb p_A^2 \\
& + \lb \int_{\mbR_+}  z_1^2  I_{ \{(z_1, 0): z_1>1\} } \ c\nu_{\ag}(dz_1)  \rb p_A(1-p_A) + \lb \int_{\mbR_+}  z_2^2  I_{ \{(z_1, 0): z_2>1\} } \ c\nu_{\ag}(dz_1)  \rb (1-p_A)p_A\\
&= \frac{2c\ag}{\ag-2}p_A^2 + \frac{c\ag}{\ag-2}p_A(1-p_A) + \frac{c\ag}{\ag-2}(1-p_A)p_A= \frac{2c\ag}{\ag-2}p_A.
\end{align*}
Similarly, we can get $\sg_{Y}^2 = \frac{2c\ag}{\ag-2}p_B$.

Turning to $\sg_{XY}$, we have that
\begin{align*}
&\sg_{XY} \\
&= \lim_{n \to \infty} n E \lb  \frac{A_1B_1Z_1^2 + A_2B_2Z_2^2}{b(n)^2} I_{(A_1^2Z_1^2 + A_2^2Z_2^2)^{1/2} \vee (B_1^2Z_1^2 + B_2^2Z_2^2)^{1/2} > b(n)} \rb \\
&= \lim_{n \to \infty} n E \lb  \frac{Z_1^2 + Z_2^2}{b(n)^2} I_{(Z_1^2 + Z_2^2)^{1/2} > b(n)} \rb \times p_A^2p_B^2 +\lim_{n \to \infty} n E \lb  \frac{Z_1^2}{b(n)^2} I_{|Z_1| > b(n)} \rb \times p_Ap_B(1-p_Ap_B) \\
& + \lim_{n \to \infty} n E \lb  \frac{Z_2^2}{b(n)^2} I_{|Z_2|  > b(n)} \rb \times (1-p_Ap_B)p_Ap_B\\
& = \lb \int_{\mbR_+}  z_1^2  I_{ \{(z_1, 0): z_1>1\} } \ c\nu_{\ag}(dz_1) + \int_{\mbR_+}  z_2^2 I_{ \{(0, z_2): z_2>1 \}} \ c\nu_{\ag}(dz_2) \rb p_A^2p_B^2 \\
& + \lb \int_{\mbR_+}  z_1^2  I_{ \{(z_1, 0): z_1>1\} } \ c\nu_{\ag}(dz_1)  \rb p_Ap_B(1-p_Ap_B) + \lb \int_{\mbR_+}  z_2^2  I_{ \{(0, z_2): z_2>1\} } \ c\nu_{\ag}(dz_2)  \rb (1-p_Ap_B)p_Ap_B\\
&= \frac{2c\ag}{\ag-2}p_A^2p_B^2 + \frac{c\ag}{\ag-2}p_Ap_B(1-p_Ap_B) + \frac{c\ag}{\ag-2}(1-p_Ap_B)p_Ap_B\\
&= \frac{2c\ag}{\ag-2}p_Ap_B.
\end{align*}
\end{proof}

\clearpage

\appendixfour
\section{Supplementary simulation results} \label{s:a_sim}

\subsection{Simulation results on the consistency of $\hat{\rho}_{n,k}$} \label{ss:a_sim}

This section reports the magnitude of empirical biases, measured by the absolute difference between the average and the theoretical value, along with standard errors computed as the sample standard deviations. Using the optimal $k$s selected by the method from~\citet{danielsson:Ergun:deHaan:deVries:2016}, the results are shown in Tables~\ref{t:sim3} and~\ref{t:sim4} for $\alpha =  \{4,5\}$. The results with the optimal $k$s selected by the method from~\citet{clauset:shalizi:newman:2009} are provided in  Tables~\ref{t:sim5},~\ref{t:sim6}, and~\ref{t:sim7}, for $\alpha =  \{3,4,5\}$.

In general, the estimators obtained with the method of \citet{danielsson:Ergun:deHaan:deVries:2016} exhibit substantially lower bias,  but larger standard errors compared to those obtained with the method of \citet{clauset:shalizi:newman:2009}. The lower bias is likely due to the fact that tail quantiles are particularly sensitive to small changes in probabilities. By minimizing the KS distance between the empirical and theoretical tail quantiles, as done by~\citet{danielsson:Ergun:deHaan:deVries:2016}, the method  appears to achieve lower bias in finite samples. The larger standard errors result from this method selecting a much smaller value of $k$ compared to \citet{clauset:shalizi:newman:2009}. In terms of MSE, no substantial difference appears to exist between the two methods.

\begin{table}[ht]
\centering
\caption{The magnitude of empirical biases (standard errors) of $\hat{\rho}_{n,k}$ when $\ag=4$. Optimal $k$s are selected using the method from~\citet{danielsson:Ergun:deHaan:deVries:2016}, with averages of $k=9$ ($N=100$), $k=29$ ($N=500$), and  $k=74$ ($N=2000$).
\label{t:sim3}}
\medskip
\begin{tabular}{cccc}
  \hline\hline
 $\rho_{XY}$ & $N=100$ &  $N=500$ & $N=2000$  \\
  \hline
-1.0 & 0.08 (0.04) & 0.06 (0.03) & 0.04 (0.03) \\
-0.9 & 0.13 (0.11) & 0.10 (0.09) & 0.08 (0.08) \\
-0.8 & 0.12 (0.14) & 0.10 (0.11) & 0.07 (0.10) \\
-0.7 & 0.10 (0.15) & 0.07 (0.12) & 0.05 (0.11) \\
-0.6 & 0.07 (0.15) & 0.06 (0.12) & 0.04 (0.11) \\
-0.5 & 0.05 (0.15) & 0.04 (0.12) & 0.02 (0.12) \\
-0.4 & 0.02 (0.15) & 0.03 (0.11) & 0.02 (0.12) \\
-0.3 & 0.01 (0.13) & 0.01 (0.11) & 0.00 (0.11) \\
-0.2 & 0.00 (0.12) & 0.00 (0.09) & 0.00 (0.07) \\
-0.1 & 0.00 (0.10) & 0.00 (0.07) & 0.01 (0.06) \\
0.0 & 0.00 (0.09) & 0.00 (0.05) & 0.00 (0.04) \\
0.1 & 0.00 (0.10) & 0.00 (0.07) & 0.00 (0.06) \\
0.2 & 0.01 (0.12) & 0.00 (0.09) & 0.00 (0.08) \\
0.3 & 0.01 (0.13) & 0.01 (0.11) & 0.01 (0.10) \\
0.4 & 0.03 (0.14) & 0.01 (0.12) & 0.01 (0.11) \\
0.5 & 0.04 (0.16) & 0.03 (0.12) & 0.02 (0.11) \\
0.6 & 0.08 (0.15) & 0.05 (0.12) & 0.03 (0.12) \\
0.7 & 0.10 (0.14) & 0.07 (0.12) & 0.05 (0.11) \\
0.8 & 0.12 (0.13) & 0.09 (0.10) & 0.07 (0.10) \\
0.9 & 0.13 (0.11) & 0.10 (0.08) & 0.08 (0.09) \\
1.0 & 0.09 (0.05) & 0.06 (0.03) & 0.04 (0.03) \\
\hline\hline
\end{tabular}
\end{table}

\begin{table}[ht]
\centering
\caption{The magnitude of empirical biases (standard errors) of $\hat{\rho}_{n,k}$ when $\ag=5$. Optimal $k$s are selected using the method from~\citet{danielsson:Ergun:deHaan:deVries:2016}, with averages of $k=9$ ($N=100$), $k=29$ ($N=500$), and  $k=79$ ($N=2000$).
\label{t:sim4}}
\medskip
\begin{tabular}{cccc}
  \hline\hline
 $\rho_{XY}$ & $N=100$ &  $N=500$ & $N=2000$  \\
  \hline
-1.0 & 0.14 (0.06) & 0.10 (0.04) & 0.08 (0.04) \\
-0.9 & 0.20 (0.12) & 0.17 (0.09) & 0.13 (0.08) \\
-0.8 & 0.19 (0.14) & 0.16 (0.10) & 0.12 (0.10) \\
-0.7 & 0.16 (0.14) & 0.13 (0.10) & 0.10 (0.10) \\
-0.6 & 0.12 (0.14) & 0.10 (0.11) & 0.08 (0.11) \\
-0.5 & 0.09 (0.14) & 0.07 (0.11) & 0.06 (0.10) \\
-0.4 & 0.06 (0.14) & 0.05 (0.09) & 0.04 (0.09) \\
-0.3 & 0.04 (0.13) & 0.03 (0.09) & 0.02 (0.09) \\
-0.2 & 0.01 (0.12) & 0.02 (0.08) & 0.01 (0.06) \\
-0.1 & 0.00 (0.11) & 0.01 (0.06) & 0.00 (0.06) \\
0.0 & 0.01 (0.10) & 0.00 (0.06) & 0.00 (0.04) \\
0.1 & 0.01 (0.11) & 0.01 (0.07) & 0.01 (0.06) \\
0.2 & 0.02 (0.12) & 0.01 (0.08) & 0.00 (0.07) \\
0.3 & 0.03 (0.13) & 0.03 (0.09) & 0.02 (0.08) \\
0.4 & 0.06 (0.13) & 0.05 (0.09) & 0.04 (0.09) \\
0.5 & 0.09 (0.14) & 0.07 (0.10) & 0.06 (0.09) \\
0.6 & 0.13 (0.14) & 0.10 (0.11) & 0.08 (0.11) \\
0.7 & 0.16 (0.13) & 0.13 (0.10) & 0.10 (0.10) \\
0.8 & 0.19 (0.13) & 0.15 (0.10) & 0.13 (0.09) \\
0.9 & 0.21 (0.12) & 0.17 (0.09) & 0.14 (0.08) \\
1.0 & 0.14 (0.06) & 0.11 (0.04) & 0.08 (0.04) \\
\hline\hline
\end{tabular}
\end{table}

\begin{table}[ht]
\centering
\caption{The magnitude of empirical biases (standard errors) of $\hat{\rho}_{n,k}$ when $\ag=3$. Optimal $k$s are selected using the method from~\citet{clauset:shalizi:newman:2009}, with averages of $k=60$ ($N=100$), $k=273$ ($N=500$), and  $k=991$ ($N=2000$).
\label{t:sim5}}
\medskip
\begin{tabular}{cccc}
  \hline\hline
 $\rho_{XY}$ & $N=100$ &  $N=500$ & $N=2000$  \\
  \hline
-1.0 & 0.10 (0.03) & 0.10 (0.02) & 0.09 (0.02) \\
-0.9 & 0.15 (0.06) & 0.14 (0.05) & 0.13 (0.04) \\
-0.8 & 0.13 (0.08) & 0.12 (0.06) & 0.12 (0.05) \\
-0.7 & 0.11 (0.08) & 0.10 (0.06) & 0.10 (0.04) \\
-0.6 & 0.09 (0.08) & 0.09 (0.06) & 0.08 (0.04) \\
-0.5 & 0.07 (0.09) & 0.06 (0.06) & 0.06 (0.04) \\
-0.4 & 0.05 (0.08) & 0.05 (0.05) & 0.05 (0.04) \\
-0.3 & 0.03 (0.07) & 0.03 (0.05) & 0.03 (0.03) \\
-0.2 & 0.02 (0.06) & 0.02 (0.03) & 0.02 (0.02) \\
-0.1 & 0.01 (0.04) & 0.01 (0.02) & 0.01 (0.01) \\
0.0 & 0.00 (0.04) & 0.00 (0.01) & 0.00 (0.01) \\
0.1 & 0.01 (0.04) & 0.01 (0.02) & 0.01 (0.01) \\
0.2 & 0.01 (0.06) & 0.02 (0.04) & 0.02 (0.02) \\
0.3 & 0.03 (0.07) & 0.03 (0.04) & 0.03 (0.03) \\
0.4 & 0.05 (0.08) & 0.05 (0.05) & 0.04 (0.04) \\
0.5 & 0.07 (0.08) & 0.07 (0.05) & 0.06 (0.04) \\
0.6 & 0.09 (0.09) & 0.09 (0.06) & 0.08 (0.05) \\
0.7 & 0.11 (0.09) & 0.11 (0.06) & 0.10 (0.04) \\
0.8 & 0.13 (0.08) & 0.13 (0.05) & 0.12 (0.04) \\
0.9 & 0.14 (0.07) & 0.14 (0.05) & 0.13 (0.04) \\
1.0 & 0.10 (0.03) & 0.10 (0.02) & 0.09 (0.02) \\
\hline\hline
\end{tabular}
\end{table}

\begin{table}[ht]
\centering
\caption{The magnitude of empirical biases (standard errors) of $\hat{\rho}_{n,k}$ when $\ag=4$. Optimal $k$s are selected using the method from~\citet{clauset:shalizi:newman:2009}, with averages of $k=53$ ($N=100$), $k=226$ ($N=500$), and  $k=791$ ($N=2000$).
\label{t:sim6}}
\medskip
\begin{tabular}{cccc}
  \hline\hline
 $\rho_{XY}$ & $N=100$ &  $N=500$ & $N=2000$  \\
  \hline
-1.0 & 0.15 (0.03) & 0.15 (0.02) & 0.14 (0.02) \\
-0.9 & 0.22 (0.06) & 0.22 (0.03) & 0.21 (0.03) \\
-0.8 & 0.20 (0.06) & 0.20 (0.04) & 0.19 (0.03) \\
-0.7 & 0.17 (0.06) & 0.17 (0.04) & 0.16 (0.03) \\
-0.6 & 0.14 (0.06) & 0.14 (0.04) & 0.13 (0.02) \\
-0.5 & 0.11 (0.06) & 0.10 (0.03) & 0.10 (0.02) \\
-0.4 & 0.08 (0.06) & 0.07 (0.03) & 0.07 (0.02) \\
-0.3 & 0.05 (0.06) & 0.05 (0.03) & 0.05 (0.02) \\
-0.2 & 0.03 (0.05) & 0.03 (0.02) & 0.03 (0.01) \\
-0.1 & 0.01 (0.05) & 0.01 (0.02) & 0.01 (0.01) \\
0.0 & 0.00 (0.04) & 0.00 (0.02) & 0.00 (0.01) \\
0.1 & 0.01 (0.05) & 0.01 (0.02) & 0.01 (0.01) \\
0.2 & 0.03 (0.05) & 0.03 (0.02) & 0.03 (0.01) \\
0.3 & 0.05 (0.05) & 0.05 (0.03) & 0.05 (0.02) \\
0.4 & 0.07 (0.06) & 0.08 (0.03) & 0.07 (0.02) \\
0.5 & 0.11 (0.06) & 0.10 (0.03) & 0.10 (0.02) \\
0.6 & 0.14 (0.06) & 0.14 (0.04) & 0.13 (0.02) \\
0.7 & 0.17 (0.06) & 0.17 (0.04) & 0.16 (0.03) \\
0.8 & 0.20 (0.06) & 0.20 (0.04) & 0.19 (0.03) \\
0.9 & 0.22 (0.06) & 0.22 (0.03) & 0.21 (0.03) \\
1.0 & 0.15 (0.03) & 0.15 (0.02) & 0.14 (0.02) \\
\hline\hline
\end{tabular}
\end{table}

\begin{table}[ht]
\centering
\caption{The magnitude of empirical biases (standard errors) of $\hat{\rho}_{n,k}$ when $\ag=5$. Optimal $k$s are selected using the method from~\citet{clauset:shalizi:newman:2009}, with averages of $k=48$ ($N=100$), $k=187$ ($N=500$), and  $k=602$ ($N=2000$).
\label{t:sim7}}
\medskip
\begin{tabular}{cccc}
  \hline\hline
 $\rho_{XY}$ & $N=100$ &  $N=500$ & $N=2000$  \\
  \hline
-1.0 & 0.19 (0.03) & 0.19 (0.02) & 0.18 (0.02) \\
-0.9 & 0.28 (0.06) & 0.28 (0.03) & 0.27 (0.02) \\
-0.8 & 0.25 (0.06) & 0.26 (0.03) & 0.25 (0.02) \\
-0.7 & 0.22 (0.06) & 0.21 (0.03) & 0.21 (0.02) \\
-0.6 & 0.17 (0.06) & 0.17 (0.03) & 0.17 (0.02) \\
-0.5 & 0.14 (0.06) & 0.13 (0.03) & 0.13 (0.02) \\
-0.4 & 0.10 (0.06) & 0.10 (0.03) & 0.10 (0.02) \\
-0.3 & 0.06 (0.06) & 0.07 (0.03) & 0.06 (0.02) \\
-0.2 & 0.04 (0.05) & 0.04 (0.03) & 0.04 (0.01) \\
-0.1 & 0.02 (0.05) & 0.02 (0.03) & 0.02 (0.01) \\
0.0 & 0.00 (0.05) & 0.00 (0.02) & 0.00 (0.01) \\
0.1 & 0.01 (0.05) & 0.02 (0.03) & 0.02 (0.01) \\
0.2 & 0.04 (0.05) & 0.04 (0.03) & 0.04 (0.01) \\
0.3 & 0.07 (0.05) & 0.07 (0.03) & 0.06 (0.01) \\
0.4 & 0.10 (0.06) & 0.10 (0.03) & 0.10 (0.02) \\
0.5 & 0.14 (0.06) & 0.14 (0.03) & 0.13 (0.02) \\
0.6 & 0.18 (0.06) & 0.17 (0.03) & 0.17 (0.02) \\
0.7 & 0.22 (0.06) & 0.22 (0.03) & 0.21 (0.02) \\
0.8 & 0.25 (0.06) & 0.25 (0.03) & 0.25 (0.02) \\
0.9 & 0.28 (0.05) & 0.28 (0.03) & 0.27 (0.02) \\
1.0 & 0.19 (0.03) & 0.19 (0.02) & 0.18 (0.02) \\
\hline\hline
\end{tabular}
\end{table}

\clearpage

\subsection{Relation to the measures $\chi$ and $\bar{\chi}$ } \label{ss:a_chi}

The joint distribution of $\|X\|$ and $\|Y\|$ can be used
to assess the likelihood of extreme curves $X$ and $Y$ occurring simultaneously, where ``extreme" refers to their size measured
by the norm.  Since $\|X\|$ and $\|Y\|$  are scalars,
we can apply to them two commonly used extremal measures $\chi$ and $\bar{\chi}$ introduced by~\citet{coles:heffernan:tawn:1999}.
This allows us to determine
whether extreme $\|X\|$ and $\|Y\|$ values occur simultaneously.

We start by recalling the definitions of $\chi$ and $\bar{\chi}$.
Let $F_U$ and $F_V$ are the marginal distribution functions of nonnegative random variables $U$ and $V$. The measure $\chi$ is defined as $\chi = \lim_{q \to 1} \chi(q)$, where
\[
\chi(q) = P(F_U(U)>q| F_V(V)>q), \quad 0<q<1.
\]
If $U$ and $V$ are asymptotically independent, then $\chi = 0$, and if they are asymptotically dependent, then $\chi \in (0,1]$. The measure $\bar{\chi}$ is defined as $\bar{\chi} = \lim_{q \to 1} \bar{\chi}(q)$, where
\[
\bar{\chi}(q) = \frac{2 \log P(F_U(U)>q)}{\log P(F_U(U)>q, F_V(V)>q)} -1, \quad 0<q<1.
\]
If $U$ and $V$ are asymptotically independent, then $\bar{\chi} \in [-1,1)$, and if they are asymptotically dependent, then $\bar{\chi} = 1$. These two measures are thus complementary, and it is useful
to apply them together. We will demonstrate that our extremal correlation coefficient is complementary to them because it provides additional information
on the shapes of the extremal curves.

We generate random curves $X$ and $Y$ as described in equation~\eqref{eq:XY}, with $N=1000$, for $\rho_{XY} \in \{0, 0.4, 0.7, 1\}$, and compute $\chi(q)$ and $\bar{\chi}(q)$ using  $(\|X\|, \|Y\|)$. We did not consider the case when  $\rho_{XY}$ is negative, as the results are similar to those for $|\rho_{XY}|$. The results are presented in Figure~\ref{f:Chi} for each value of $\rho_{XY}$. When $\rho_{XY}=0$, extreme curves in $X$ and $Y$ do not occur simultaneously, and the corresponding values for $\chi$ and $\bar{\chi}$ are both close to zero, see the upper left paired plot in Figure~\ref{f:Chi}. When $\rho_{XY} \in \{0.4, 0.7, 1\}$, extreme curves in $X$ and $Y$ tend to occur simultaneously. In that case, $\chi$ should be greater than 0 and $\bar{\chi}(q)$ should approach 1 as $q \to 1$, which is observed in the other paired plots in Figure~\ref{f:Chi}.

\begin{figure}[h]
\begin{centering}
\includegraphics[width=1\textwidth]{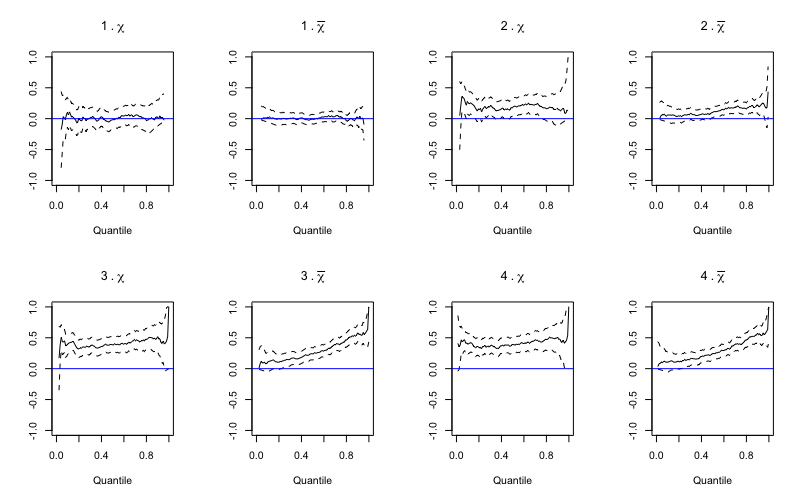}
\par\end{centering}
\caption{The values of $\chi(q)$ and $\bar{\chi}(q)$ (solid) with 95\% confidence bands (dashed) are displayed in the paired plots. The values are computed from $\|X\|$ and $\|Y\|$, where $X$ and $Y$ are defined as in~\eqref{eq:XY}. The paired plots are arranged for $\rho_{XY} = 0$ in the upper left, $0.4$ in the upper right, $0.7$ in the lower left, and $1$ in the lower right.
\label{f:Chi}}
\end{figure}

It is important to emphasize that the measures $\chi$ and $\bar{\chi}$ for $(\|X\|$, $\|Y\|$) should be used with the extremal correlation coefficient $\rho_{XY}$. Although $\chi$ and $\bar{\chi}$ quantify whether extreme curves in $X$ and $Y$ occur simultaneously, they do not account for the shapes of curves. Therefore, $\rho_{XY}$ complements $\chi$ and $\bar{\chi}$ by evaluating the relationship between the shapes of the curves, highlighting the new, functional aspect of $\rho_{XY}$.
To illustrate this, consider the following toy example:
\begin{equation} \label{eq:XY_1}
X(t) = Z_{1}\phi_1(t) + N_1\phi_2(t); \ \
Y(t) = Z_{1}\phi_2(t) + N_2\phi_1(t),
\end{equation}
where $Z_1$, $N_1$, $N_2$, $\phi_1$, and $\phi_2$ are defined in Lemma~\ref{l:sim1}. Since $X$ and $Y$ share $Z_1$, extreme events occur in both $X$ and $Y$ simultaneously. It then follows that $\chi=\bar{\chi}=1$, as shown in Fig~\ref{f:Chi_1}. However, when examining the extreme curves in $X$ and $Y$, their patterns are unrelated since $\phi_1$ and $\phi_2$ are orthogonal. In this case, $\rho_{XY}$ captures the lack of similarity between the shapes, resulting in an estimate close to 0.

\begin{figure}[h]
\begin{centering}
\includegraphics[width=.5\textwidth]{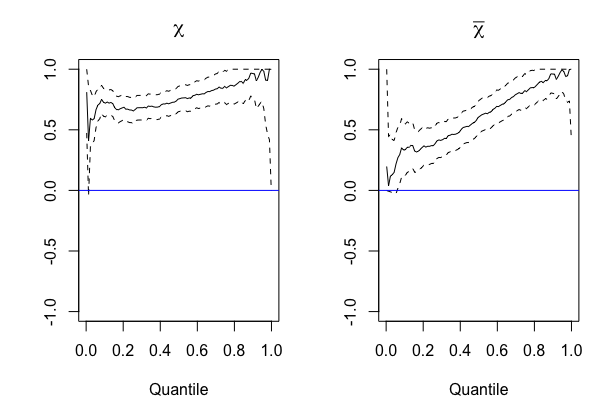}
\par\end{centering}
\caption{The values of $\chi(q)$ and $\bar{\chi}(q)$ (solid) with 95\% confidence bands (dashed) are displayed in the paired plot. The values are computed from $\|X\|$ and $\|Y\|$, where $X$ and $Y$ are defined as in~\eqref{eq:XY_1}.
\label{f:Chi_1}}
\end{figure}

\clearpage

\subsection{Effect of phase variation on $\rho_{XY}$} \label{ss:a_phv}

Phase variation occurs when some  properties of curves shift over time, as seen with growth spurts occurring at different times for different individuals. Any variation in phase can affect the extremal correlation coefficient. In general, if extremal curves in the two samples are
out of phase, the sample extremal correlation coefficient will become closer to zero.
For example, if heat
waves tend to arrive at different times at different locations,
this will result in the extremal correlation coefficient closer to zero than if the heat wave
arrival times matched.
To illustrate this, we consider the data generating process described in~\eqref{eq:XY}, but with $\phi_k$ in $Y(t)$, replaced by
\[
\phi_k^*(t) = 0, \ {\rm if} \  t\le 0.3, \quad
\phi_k^*(t) = \phi_k\lp t-0.3\rp, \ {\rm if} \  t> 0.3.
\]
Table~\ref{t:sim8} presents the $\hat{\rho}_{n,k}$
with and without phase shift. The results indicate that phase shift
brings the $\hat{\rho}_{n,k}$ closer to zero.

\begin{table}[ht]
\centering
\caption{The $\hat{\rho}_{n,k}$ for samples without and with phase variation when $\ag=3$ and $N=100$. The cut-off $k$s is selected using the method of~\citet{danielsson:Ergun:deHaan:deVries:2016}.
\label{t:sim8}}
\medskip
\begin{tabular}{ccc}
  \hline\hline
   & \multicolumn{2}{c}{Phase variation}  \\ \cline{2-3}
 $\rho_{XY}$ & No &  Yes \\
  \hline
-1.0 & -0.96 & -0.90\\
-0.9 & -0.84 & -0.78\\
-0.8 & -0.73 & -0.69\\
-0.7 & -0.66 & -0.62\\
-0.6 & -0.57 & -0.54\\
-0.5 & -0.49 & -0.47\\
-0.4 & -0.41 & -0.39\\
-0.3 & -0.31 & -0.29\\
-0.2 & -0.22 & -0.20\\
-0.1 & -0.11 & -0.09\\
0.0 & 0.00 & -0.02 \\
0.1 & 0.11 & 0.10 \\
0.2 & 0.21 & 0.20 \\
0.3 & 0.32 & 0.28 \\
0.4 & 0.40 & 0.38 \\
0.5 & 0.50 & 0.46 \\
0.6 & 0.56 & 0.53 \\
0.7 & 0.66 & 0.61 \\
0.8 & 0.74 & 0.69 \\
0.9 & 0.84 & 0.78 \\
1.0 & 0.96 & 0.90 \\
\hline\hline
\end{tabular}
\end{table}

\clearpage

\appendixfive
\section{Supplementary results for Section~\ref{s:dt}} \label{s:a_dt}

This section presents Hill plots for the CIDRs ETF data discussed in Section~\ref{ss:etf} and the temperature data discussed in Section~\ref{ss:hw}. These plots suggest that the extreme curves discussed in those sections appear to be regularly varying, as the plots generally show stable regions.

\begin{figure}[h]
\begin{centering}
\includegraphics[width=1\textwidth]{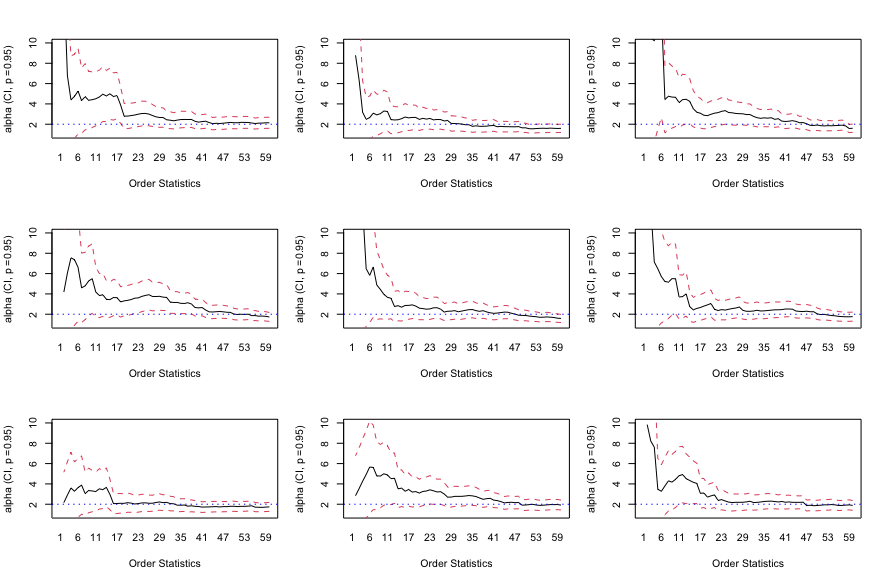}
\par\end{centering}
\caption{Hill plots of the norm of centered CIDRs for each  sector ETF, with the Hill estimates (solid) and 95\% confidence intervals (dashed). From left to right, the upper row shows: XLY, XLP, XLE; the middle row: XLF, XLV, XLI; and the lower row: XLB, XLK, XLU.
\label{f:Hill_e}}
\end{figure}

\begin{figure}[h]
\begin{centering}
\includegraphics[width=1\textwidth]{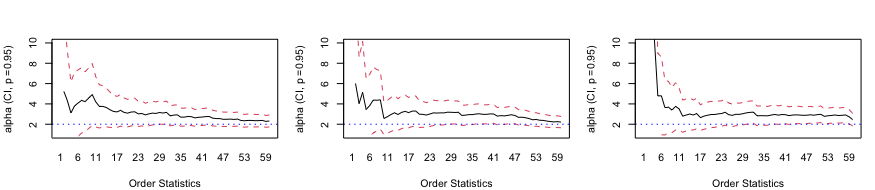}
\par\end{centering}
\caption{Hill plots of the norm of centered temperature curves for each location, with the Hill estimates (solid) and 95\% confidence intervals (dashed). From left to right, Fort Collins, CO, Colorado Springs, CO, and Austin, TX.
\label{f:Hill_t}}
\end{figure}

\end{document}